\documentclass[reqno,a4paper,12pt]{amsart}
\usepackage{amssymb}
\usepackage{graphicx, color, hyperref}
\usepackage[utf8]{inputenc}

\allowdisplaybreaks

\newtheorem{theorem}{Theorem}
\newtheorem{lemma}{Lemma}
\newtheorem{proposition}{Proposition}
\newtheorem{corollary}{Corollary}
\newtheorem{conjecture}{Conjecture}

\theoremstyle{definition}
\newtheorem{definition}{Definition}
\newtheorem{claim}{Claim}

\theoremstyle{remark}
\newtheorem{remark}{Remark}

\newcommand{\sepdist}{\mathrm{sep\_dist}}

\newcommand{\Z}{\mathbb{Z}}
\newcommand{\C}{\mathbb{C}}
\newcommand{\N}{\mathbb{N}}
\newcommand{\R}{\mathbb{R}}
\renewcommand{\S}{\mathbb{S}}
\newcommand{\Sd}{\mathbb{S}^d}
\renewcommand{\l}{\lambda}

\newcommand{\K}{\mathcal{K}}
\newcommand{\tr}{\operatorname{tr}}
\newcommand{\Gammaf}[2]{\Gamma\left(\frac{#1}{#2}\right)}

\begin{document}

\title[Energy of determinantal point processes in the sphere]{Energy and 
discrepancy of rotationally invariant determinantal point processes in high 
dimensional spheres}

\author{Carlos Beltr\'an, Jordi Marzo and Joaquim Ortega-Cerd\`a}

%\author{Carlos Beltr\'an}
\address{ Departamento de Matematicas, Estad\'{\i}stica y Computaci\'on, 
Universidad de Can\-ta\-bria, Avd. Los Castros s/n, 39005, Santander, Spain}
\email{beltranc@unican.es}

%\author{Jordi Marzo}
\address{Departament de Matem\`atiques i Inform\`atica, Universitat de 
Barcelona \& Barcelona Graduate School of Mathematics, Gran Via 585, 08007, 
Barcelona, Spain}
\email{jmarzo@ub.edu}

%\author{Joaquim Ortega-Cerd\`a}
\address{Departament de Matem\`atiques i Inform\`atica, Universitat de 
Barcelona \& Barcelona Graduate School of Mathematics, Gran Via 585, 08007, 
Barcelona, Spain}
\email{jortega@ub.edu}

\thanks{This research has been partially supported by MTM2014-57590-P and 
MTM2014-51834-P grants by the Ministerio de Econom\'{\i}a y Competitividad, 
Gobierno de Espa\~na and by the Ge\-ne\-ra\-li\-tat de Catalunya (project 2014 
SGR 289).}

\begin{abstract}
We study expected Riesz $s$-energies and linear statistics of some determinantal 
processes on the sphere $\S^d$. In particular, we compute the expected Riesz and 
logarithmic energies of the determinantal processes given by the reproducing 
kernel of the space of spherical harmonics. This kernel defines the so called 
harmonic ensemble on $\S^d$. With these computations we improve previous 
estimates for the discrete minimal energy of configurations of points in the 
sphere. We prove a comparison result for Riesz $2$-energies of points defined 
through determinantal point processes associated to isotropic kernels. As a 
corollary we get that the Riesz $2$-energy of the harmonic ensemble is optimal 
among ensembles defined by isotropic kernels with the same trace. Finally, we 
study the variance of smooth and rough linear statistics for the harmonic 
ensemble and compare the results with the variance for the spherical ensemble 
(in $\S^2$). 
\end{abstract}

\maketitle

\section{Introduction}\label{sec:introduction}

Let $\S^d$ be the unit sphere in the Euclidean space $\R^{d+1}$ and let $\mu$ be the normalized Lebesgue surface measure. We study Riesz 
$s$-energies and the uniformity (discrepancy and separation) of random 
configurations of points on the sphere $\S^d$ given by some determinantal point 
processes.

\subsection{Riesz energies}
For a given collection of points $x_1,\ldots,x_n\in\S^d$ and $s>0$ the discrete 
$s$-energy associated to the set $x=(x_1,\ldots,x_n)$ is 
\[
E_s(x)=\sum_{i\neq j}\frac{1}{\|x_i-x_j\|^s}.
\]
The minimal Riesz $s$-energy is the value 
$\mathcal{E}(s,n)=\inf_{x} E_s(x)$, where $x$ 
runs on the $n$-point subsets of $\S^d$. The limiting case $s=0$ given (through 
$(t^s-1)/s\rightarrow\log t$ when $s\to 0$) by
\[
E_0(x)=\sum_{i\neq j}\log \frac{1}{\|x_i-x_j\|},
\]
is the discrete logarithmic energy associated to $x$ and 
$\mathcal{E}(0,n)=\inf_{x} E_0(x)$, is the minimal 
discrete logarithmic energy for $n$ points on the sphere.

The asymptotic behavior of these energies, in the spherical and in other 
settings, has been extensively studied. See for example the survey papers 
\cite{BHS12,BG15}. One problem in studying these quantities is to get computable 
examples. It is natural then to study random configurations of points and try to 
estimate asymptotically their energies on average. The next natural question is 
how to get good random configurations on the sphere. It is clear that uniformly 
random points are not good candidates to have low energies because there is no 
local repulsion between points and the sets exhibit clumping. 

A method to get better distributed random points is to take sets of 
zeros of certain random holomorphic polynomials on the plane with independent 
coefficients and transport them to the sphere via the stereographic projection. 
As the zeros repel each other, the configurations exhibit no clumping. This idea 
was used in \cite{ABS11} and the authors managed to get, in $\S^2$, the average 
behavior of the logarithmic energy (other relations between the logarithmic 
energy and polynomial roots are known, see \cite{ShSm93}). See the works of 
Zelditch et al. \cite{SZ08,ZZ08,FZ13} for an extension to several complex 
variables. We consider instead random sets of points given by a determinantal 
process. Random points drawn from a determinantal process exhibit local 
repulsion, they can be built in any dimension and they are computationally 
feasible, as proven in \cite{GAF} and implemented in \cite{SZT09}.

\subsection{Determinantal processes}

In this section we follow \cite[Chap.~4]{GAF}. See also \cite{AGZ,SZT09} or \cite{S00}.

We denote as $\mathcal{X}$ a (simple) random point process in $\S^d$. A way to 
describe the process is to specify the random variable counting the number of points of the process in $D,$
for all Borel sets $D\subset \S^d$. We denote this random variable as $\mathcal{X}(D)$ or $n_D.$
In many cases the point process is 
conveniently characterized by the so called joint intensity functions, see 
\cite{Lenard73, Lenard75}.

The joint intensities $\rho_k(x_1,\dots, x_k)$ are functions defined in 
$(\S^d)^k$ such 
that for any family of mutually disjoint subsets $D_1,\dots ,D_k \subset \S^d$
\[
\mathbb{E}\left[ \mathcal{X}(D_1)\cdots \mathcal{X}(D_k) \right] 
=\int_{D_1\times \dots \times D_k} \rho_k(x_1,\dots, x_k) d\mu(x_1)\dots 
d\mu(x_k),
\]
we assume that $\rho_k(x_1,\dots, x_k)=0$ when $x_i=x_j$ for $i\neq j$.

A random point process on the sphere is called determinantal with kernel 
$K:\S^d\times \S^d \rightarrow \C,$ if it is simple and the joint intensities
with respect to a background measure $\mu$ (the normalized surface measure in 
our case) are given by
\[
\rho_k(x_1,\dots, x_k)=\det (K(x_i,x_j))_{1\le i,j\le k}, 
\]
for every $k\ge 1$ and $x_1,\dots, x_k \in \S^d$.

We will mostly restrict ourselves to a special class of determinantal point 
processes, induced by the so called projection kernels.
\begin{definition}
 We say that $K$ is a projection kernel if it is a Hermitian projection kernel, 
i.e. the integral operator in $L^2(\mu)$ with kernel $K$ is selfadjoint and has 
eigenvalues $1$ and $0$.
\end{definition}

A projection kernel $K(x,y)$ defines a determinantal process 
with $n$ points a.s. if its trace equals $n$, i.e.
\[
 \int_{\S^d}K(x,x)d\mu(x)=n.
\]
In this case, the random vector in $(\S^d)^n$ with density $\frac{1}{n!}\det 
(K(x_i,x_j))_{1\le i,j\le k}$ is a determinantal process with the right 
marginals i.e. the joint intensities are given by determinants of the kernel 
\cite[Remark~4.2.6]{AGZ}.

Determinantal processes on $\S^2$ have been considered before. In \cite{ASZ14} 
the authors study the so called spherical ensemble. The points of this process 
correspond to the generalized eigenvalues of certain random matrices, mapped to 
the surface of the sphere by the stereographic projection. It was shown by 
Krishnapur \cite{Kri09} that this process is determinantal and the kernel is the 
reproducing kernel of a weighted space of polynomials on the plane. In 
\cite{ASZ14} the authors obtain, among other results, the expected Riesz 
logarithmic and $s$-energies (for $s$ in some range) and they use it to improve 
previous bounds for the minimal discrete energies. This point process does not 
have an immediate extension to arbitrary dimensions. 

To compute the expected energy of a determinantal process we use the following 
well known result:

\begin{proposition}\label{th:determinantal}
A projection kernel $K,$ with trace $n,$ defines a determinantal point process 
on $\S^d$ which generates $n$ points at random in $\S^d$. Moreover, let 
$x=(x_1,\ldots,x_n)$ be generated by the associated point process. Then, for any 
measurable $f:\S^d\times\S^d\rightarrow[0,\infty)$ we have
\[
 \mathbb{E}_{x\in (\S^d)^n}\left(\sum_{i\neq j}f(x_i,x_j)\right)= 
\int_{x,y\in\S^d}\left(K(x,x)K(y,y)-|K(x,y)|^2\right)f(x,y)\,
d\mu (x)\,d\mu (y).
\]
\end{proposition}

The fact that $K$ defines a determinantal point process in $\S^d$ is granted by Macchi-Soshnikov's theorem
\cite[Theorem 4.5.5]{GAF} and from \cite[Formula (1.2.2)]{GAF} the formula above 
follows.

We will be interested in the values of the Riesz $s$-energy and the 
logarithmic 
energy of points coming from the determinantal point process with 
a projection kernel $K$ so in particular we will use the following corollary.

\begin{corollary}\label{cor:rieszenergies}
The expected value of the Riesz $s$-energy and the logarithmic energy of $n$ 
points given by the determinantal point process associated to $K$ are given by:
 \begin{align*}
  \mathbb{E}_{x\in (\S^d)^n}(E_s(x))= 
&\int_{x,y\in\S^d}\frac{K(x,x)K(y,y)-|K(x,y)|^2}{|x-y|^s}\,d\mu 
(x)\,d\mu (y),\\
  \mathbb{E}_{x\in 
(\S^d)^n}(E_0(x))=&\int_{x,y\in\S^d}\left(K(x,x)K(y,y)-|K(x,y)|^2\right)\log 
|x-y|^{-1}\,d\mu (x)\,d\mu (y).
 \end{align*}
 In particular,  if we let $g(s)= 
\mathbb{E}_{x\in (\S^d)^n}(E_s(x))$ then $ \mathbb{E}_{x\in 
(\S^d)^n}(E_0(x))=g'(0)$.
\end{corollary}

\subsection{Spherical harmonics}
For a classical introduction to shperical harmonics see for example \cite[Ch. IV]{Stein}. Given an integer $\ell\ge 0$, let $\mathcal{H}_\ell$ be the vector space of 
spherical harmonics of degree $\ell$, and let 
$h_\ell=\operatorname{dim}\mathcal{H}_\ell$. 
It is known that
\[
h_\ell=\frac{2\ell+d-1}{\ell+d-1}\binom{\ell+d-1}{\ell}= 
\frac{2}{\Gamma(d)}\ell^{d-1}+o(\ell^{d-1}).
\]
For the Hilbert space $L^2(\S^d)$ of square integrable functions in $\S^d$
with the inner product
\[
\langle f,g \rangle=\int_{\S^d} f(x)g(x)d\mu(x),\;\;\; f,g\in L^2(\S^d),
\]
one has that $L^2(\S^d)=\bigoplus_{\ell\ge 0} \mathcal{H}_\ell$ and therefore 
the expansion in an orthonormal basis of spherical harmonics provides a 
generalization of the Fourier series.

The Gegenbauer polynomials $C^\alpha_k(t)$ are the orthogonal polynomials in 
$[-1,1]$ with respect to the weight $(1-t^2)^{\alpha-\frac{1}{2}}.$ We assume 
the normalization $C^\alpha_k(1)=\binom{\alpha+k-1}{k}.$ 
Let $\{ Y_{\ell,k} 
\}_{k=1}^{h_\ell}$ be an orthonormal basis with respect to the norm in 
$L^2(\S^d)$. 
The reproducing kernel 
$Z_\ell(x,y)$ in $\mathcal{H}_\ell$ is then 
\[
Z_\ell(x,y)=\sum_{k=1}^{h_\ell} Y_{\ell,k}(x)Y_{\ell,k}(y)=\frac{
2\ell+d-1}{d-1}C^{\frac{d-1}{2}}_\ell(\langle x,y\rangle),\;\;x,y\in \S^d,
\]
where $\langle x, y\rangle$ is the scalar product in $\R^{d+1}$. 
Observe that $d(x,y)=\arccos \langle x,y \rangle$ is the geodesic distance in $\S^d.$
The function 
$Z_\ell$ is also known as the zonal harmonic of degree $\ell$.

We denote by $\Pi_L$ the vector space of spherical harmonics of degree at most 
$L$ in $\S^d$ (which equals the space of polynomials of degree at most $L$ in 
$\R^{d+1}$ restricted to $\S^d$). Its dimension is
\begin{equation}\label{eq:piL}
\dim \Pi_L=\pi_L=\frac{2L+d}{d}\binom{d+L-1}{L} = 
\frac{2}{\Gamma(d+1)}L^d+o(L^d).
\end{equation}
As the spaces $\mathcal{H}_\ell$ are mutually orthogonal one can see using the
Christoffel-Darboux formula that the reproducing kernel $K_L(x,y)$ of $\Pi_L$ is
\[
K_L(x,y)=\sum_{\ell=0}^{L}  Z_\ell(x,y)
=\frac{\pi_L}{\binom{L+\frac{d}{2}}{L}}P_L^{(1+\lambda,\lambda)}(\langle x,y 
\rangle),\;\;x,y\in \S^d,  
 \]
where $\lambda=\frac{d-2}{2}$ and the Jacobi polynomials 
$P_L^{(1+\lambda,\lambda)}(t)$ are normalized as 
\[
P_L^{(1+\lambda,\lambda)}(1)=\binom{L+\frac{d}{2}}{L}=\frac{\Gamma\left(L+\frac{
d}{2}+1\right)}{\Gamma(L+1)\Gamma\left(\frac{d}{2}+1\right)}.
\]

The following classical asymptotic estimate is a particular case from \cite[Theorem 
8.21.13]{Sze39}
\begin{equation}         				\label{Jacobiestimate} 
    P_{L}^{(1+\lambda,\lambda)}(\cos
    \theta)=\frac{k(\theta)}{\sqrt{L}}\left\{ \cos \left((L+\l+1)\theta+\gamma 
\right)+\frac{O(1)}{L\sin
    \theta}\right\},
\end{equation}
    if $c/L\le \theta \le \pi-(c/L)$, where $c$ is a fixed positive constant and
\[
k(\theta)=\pi^{-1/2} \left( \sin \frac{\theta}{2}\right)^{-\l-3/2}\left(\cos
    \frac{\theta}{2}\right)^{-\l-1/2},\quad
\gamma=-\left(\l+\frac{3}{2}\right)\frac{\pi}{2}.
\]

Near the end points, the asymptotic behavior of the Jacobi polynomials is 
given by the Mehler-Heine formulas
\begin{align}                                                    
\label{mehlerheine}                                               
    \lim_{L\to \infty} & 
L^{-1-\lambda}P_{L}^{(1+\lambda,\lambda)}\left(\cos\frac{z}{L}\right)
    =\left( \frac{z}{2}\right)^{-1-\lambda}J_{1+\lambda}(z),  \nonumber
\\
\lim_{L\to \infty} &  L^{-\lambda} P_{L}^{(1+\lambda,\lambda)} \left( 
\cos\left( \pi-\frac{z}{L}\right)\right)
    =\left( \frac{z}{2}\right)^{-\lambda}J_{\lambda}(z),
\end{align}
where the limits are uniform on compact subsets of $\C$, \cite[p.~192]{Sze39}.

By definition of reproducing kernel 
\[
P(x)=\langle P,K_L(\cdot, x) \rangle=\int_{\S^d} K_L(x,y) P(y)d\mu(y),
\] 
for $P\in \Pi_L.$ Observe that $K_L(x,x)=\pi_L$ for every $x\in\S^d$.

\begin{definition}
The harmonic ensemble is the determinantal point process in $\S^d$ with $\pi_L$ points a.s.
induced by the reproducing kernel $K_L(x,y)$. 
\end{definition}

\subsection{Linear statistics and spherical cap discrepancy}

Given a point process $\mathcal{X}$ on the sphere and a measurable function 
$\phi:\S^d\rightarrow \C,$ the corresponding linear statistic is the
random variable
\[
\mathcal{X}(\phi)=\int_{\S^d} \phi \, d\mathcal{X}.
\]
When $\phi=\chi_A$ is the characteristic function of $A\subset \S^d,$ we have that 
$\mathcal{X}(\chi_{A})$ is just the number of points of
the process $\mathcal{X}$ in $A$, also denoted as $\mathcal{X}(A)$ or $n_A.$ Characteristic 
functions define rough linear statistics, when $\phi$ is an smooth function we talk about smooth linear statistics.

% Following Beck \cite{Bec84}, see also \cite{ASZ14}, one can deduce properties 
% for the spherical cap discrepancy of a random set of points by using the 
% information from the
% rough linear statistics.

A measure of the uniformity of the distribution of a finite set of points is the spherical cap discrepancy
defined by
\[
\mathbb{D}(x)=\sup_{A} \Bigl| \frac{1}{n}\sum_{i=1}^{n} \chi_{A}(x_i)- 
\mu(A)\Bigr|,
\]
where $x=(x_1, \dots, x_n)\in (\S^d)^n$ and $A$ runs on the spherical caps (i.e. balls with respect to the geodesic distance) of $\S^d$.

It is well known that a system 
of points $\{ x^{(n)} \}_n,$ where $x^{(n)}\in (\S^d)^{m_n},$ is asymptotically 
uniformly distributed if and only if $\lim_{n\to\infty} \mathbb{D}(x^{(n)})=0.$ 
Therefore, a measure of the similarity between the atomic measures 
$\frac{1}{m_n}\sum_{i=1}^{m_n} \delta_{x_i^{(n)}}$ and the Lebesgue surface measure, $\mu,$ is 
the speed of this convergence.

It was shown by Beck in \cite{Bec84} that there exists an $n$-point set in 
$\S^d$ with spherical cap discrepancy smaller than a constant times 
$n^{-\frac{1}{2}(1+\frac{1}{d})}\sqrt{\log n}.$ To prove this result Beck uses a 
random distribution of points and the proof
is therefore non-constructive. The known explicit constructions are still far from this bound, 
see \cite{LPS86,LPS87} and \cite{Aistleitneretal}
for the $\S^2$ case.
 The random configuration used by Beck consist in taking points uniformly in each set 
of a, so called, area-regular partition of the sphere, see \cite{KS98,RSZ94}. 
For other interesting properties of this random process see \cite{BSSW14}.

The upper bound above is almost optimal because, in \cite[p.~35]{Bec84}, Beck 
shows that the spherical cap discrepancy of an $n$ point set is bounded below by
(a constant times) $n^{-\frac{1}{2}(1+\frac{1}{d})}.$  

Following Beck \cite{Bec84}, see also \cite{ASZ14}, we will deduce information 
about the spherical cap discrepancy of a random set of points drawn from the 
harmonic ensemble by using the rough linear statistic defined above. 

\subsection{Separation distance} We discuss also another measure of how 
well distributed a collection $x=(x_1,\ldots,x_n)\in (\S^d)^n$ of spherical point is, 
namely, the separation distance
\[
 \sepdist(x)=\min_{i\neq j}\|x_i-x_j\|.
\]
A well distributed collection of points should have all of its points 
well-separated, so one can search for $x$ maximizing $\sepdist(x)$. This is a 
classical problem known as the hard spheres problem, the best packing problem or 
Tammes problem since \cite{Tammes}. In the $2$-dimensional case, a surprisingly 
sharp result \cite{Waerden} is known:
\[
\min_{x=(x_1,\ldots,x_n)\in(\S^2)^n}\sepdist(x) 
=\sqrt{\frac{8\pi}{\sqrt{3}}}n^{-1/2}+O(n^{-2/3}).
\]
For the $d$-dimensional case the precise value of the constant is unknown but 
we still have that the minimal separation distance $\min_{x=(x_1,\ldots,x_n)\in(\S^d)^n} \sepdist(x)$ is of order $n^{-1/d}.$
Collections of points minimizing the Riesz energy for $s=d-1$ have been proven 
to satisfy this bound with constant $2^{1/d}$, see \cite{DragnevSaff}. 
Following \cite{ASZ14} we will obtain a bound on the separation distance of 
points choosen at random from the harmonic ensemble.

\subsection{Notation} 
For two 
sequences $x_n$, $y_n$ of positive real numbers the expressions
\[
 x_n\lesssim y_n, \;\; x_n=O(y_n)\;\;\mbox{and}\;\; y_n=\Omega(x_n),
\]
all mean that there is a constant $C\geq0$ independent of $n$ such that
\[
\limsup_{n\rightarrow\infty}\frac{x_n}{y_n}\leq C.
\]
We sometimes write the inequality in the opposite order $x_n\gtrsim y_n$ (which 
means $y_n\lesssim x_n$). Finally, if both $x_n\lesssim y_n$ and $x_n\gtrsim 
y_n$ we simply write $x_n\sim y_n$. We also recall that $x_n=o(y_n)$ means that
\[
 \limsup_{n\rightarrow\infty}\frac{x_n}{y_n}=0.
\]

\subsection*{Acknowledgements}
We want to thank the two anonymous referees for their helpful comments and 
Lambert Gaulthier for spotting an inacuracy in a previous version of the 
manuscript.

\section{Main results}\label{sec:main}

\subsection{Riesz and logarithmic energies of the harmonic ensemble}

Our first result is the computation of the expected Riesz $s$-energy (in this 
section for $0<s<d$) of the determinantal process given by the reproducing 
kernel $K_L(x,y)$ of the space of polynomials of degree at most $L$ i.e. for 
points from the harmonic ensemble. The closed expression for the energy is given 
in terms of a generalized hypergeometric function.

Recall that for integer $p,q\ge 0$ and complex values $a_i,b_j$
the generalized hypergeometric function is defined by the power series
\begin{equation}	\label{eq:qFp}
\,{}_pF_q(a_1,\ldots,a_p;b_1,\ldots,b_q;z) = \sum_{n=0}^\infty 
\frac{(a_1)_n\dots(a_p)_n}{(b_1)_n\dots(b_q)_n} \, \frac {z^n} {n!},
 \end{equation}
where $(\cdot)_n$ is the rising factorial or Pochhammer symbol given by $(x)_0=1$ for $x\in\C$ and
\[
(x)_n=x(x+1)\cdots(x+n-2)(x+n-1)=\frac{\Gamma(x+n)}{\Gamma(x)},\quad n\geq1.
\]
(Note that the formula involving the Gamma function is not defined for integer $x\leq0$, but the finite product allways is.)
%The sum in \eqref{eq:qFp} is finite in some situations, see 
%Remark~\ref{rmk:qFp}.

The continuous $s$-energy for the normalized Lebesgue measure is defined as
\begin{equation}  		\label{continuousrieszenergy}
V_s(\S^d)=\int_{\S^d} \int_{\S^d} \frac{1}{\|x-y\|^s}\, 
d\mu(x)\,d\mu(y)=2^{d-s-1}
\frac{\Gamma\left( \frac{d+1}{2} \right)\Gamma\left( \frac{d-s}{2} 
\right)}{\sqrt{\pi}\Gamma\left( d-\frac{s}{2} \right)}.
 \end{equation}
Then (recall that $\lambda=(d-2)/2$) one can write this quantity in terms of the beta 
function 
\begin{equation}\label{eq:Vs}
V_s(\S^d)=\frac{\omega_{d-1} 
2^{d-s-1}}{\omega_d}B\left(\lambda+1,\lambda+1-\frac{s}{2}\right),
\end{equation}
where 
\begin{equation}\label{eq:omegad}
\omega_d=2   \frac{\pi^{\frac{d+1}{2}}}{\Gamma\left(\frac{d+1}{2} \right)}
\end{equation}
is the surface area of $\S^d$. 

\begin{theorem}\label{th:rieszsharmonic}
 Let $x=(x_1,\ldots,x_n)\in (\S^d)^n,$ where $n=\pi_L,$ be $n$ points drawn from the harmonic ensemble. 
Then, $\mathbb{E}_{x\in(\S^d)^n}(E_s(x))$ is finite iff $s<d+2$. Moreover, for 
$0<s<d$,
\[
\begin{split}
\mathbb{E}_{x\in(\S^d)^n}(E_s(x))=n^2 V_s(\S^d)-\frac{ 2^{d-1-s}\omega_{d-1} 
n^2}{\binom{L+\frac{d}{2}}{L}^2 \omega_d}  
\frac{\Gamma\left( \frac{d-s}{2} \right) }
{ \Gamma\left( 1+\frac{d}{2} \right)\Gamma\left( 1+\frac{s}{2} \right)} 
\\
\times C_{s,d}(L) {}_{4}F_{3}\left( 
-L,d+L,\frac{d-s}{2},-\frac{s}{2};\frac{d}{2}+1,d-\frac{s}{2}+L,-\frac{s}{2}
-L;1 
 \right),
 \end{split}
 \]
where the constant
\[
C_{s,d}(L)=\frac{\Gamma\left(  L+\frac{d}{2} \right)\Gamma\left(  
L+\frac{d}{2}+1 \right) \Gamma\left( L+\frac{s}{2}+1 \right)}
{ \Gamma(L+1)^2 \Gamma\left( L-\frac{s}{2}+d \right)}\sim L^s,\;\;L\to \infty,
\]
and 
${}_{4}F_{3}$ is a generalized hypergeometric function.
\end{theorem}
The expression in Theorem \ref{th:rieszsharmonic} does not directly give us an 
insight on the dependence of the expected value with respect to the number of 
points $n$, since $L$ and $n$ are related and appear in different places of the 
formula. In order to get the asymptotic expansion of the expected Riesz 
$s$-energy we show that the generalized hypergeometric function converges to a hypergeometric function 
and then we use Gauss's theorem. To prove this convergence we 
use classical estimates of the Jacobi polynomials to get 
Proposition~\ref{pr:integralestimate} which we think may be of independent 
interest. With this result we get the following asymptotic behavior, which 
clarifies the dependence on $n$ of the formula in Theorem 
\ref{th:rieszsharmonic}.

\begin{theorem}\label{th:asymptoticrieszsharmonic}
Let $x=(x_1,\ldots,x_n)\in (\S^d)^n,$ 
where $n=\pi_L,$ be $n$ points drawn from the harmonic ensemble. Then, for $0<s<d$,
 \[
  \mathbb{E}_{x\in(\S^d)^n}(E_s(x))=V_s(\S^d) n^2 -C_{s,d} 
n^{1+s/d}+o(n^{1+s/d}),
 \]
where 
\begin{equation}		\label{constanttheorem2}
C_{s,d}=2^{s-s/d}V_s(\S^d)\left(d!\right)^{-1+\frac{s}{d}}
\frac{ d\,\Gamma\left( 1+\frac{d}{2} \right)   \Gamma\left( \frac{1+s}{2} 
\right)\Gamma\left(d-\frac{s}{2}\right) }
{ \sqrt{\pi} \Gamma\left( 1+\frac{s}{2} \right) \Gamma\left( 1+\frac{s+d}{2} 
\right)}.
 \end{equation}
\end{theorem}

The correct order of growth for the second term of the minimal Riesz $s$-energy 
is known (see \cite{KS98}) i.e.
for $d\ge 2$ and $0<s<d$ there exist constants $C,c>0$ such that
\begin{equation}		\label{knownboundsenergy}
-cn^{1+s/d}\le \mathcal{E}(s,n)-V_s(\S^d) n^2 
\le -Cn^{1+s/d}, 
\end{equation}

for $n\ge 2$.

It has been conjectured in (\cite[Conjecture 3]{BHS12}) that there is a 
constant 
$A_{s,d}$ such that
\[
\mathcal{E}(s,n)=V_s(\S^d) n^2 
+\frac{A_{s,d}}{\omega_d^{s/d}}  n^{1+s/d} +o(n^{1+s/d}).
\]
Furthermore, when $d=2,4,8,24$
\begin{equation}		\label{constantconjectured}
A_{s,d}=|\Lambda_d|^{s/d} \zeta_{\Lambda_d}(s), 
\end{equation}
where $|\Lambda_d|$ stands for the co-volume and $\zeta_{\Lambda_d}(s)$ for the 
Epstein zeta function of the lattice $\Lambda_d.$ Here $\Lambda_d$ denotes the 
hexagonal 
lattice for $d=2$, the root lattices $D_4$ for $d=4$ and $E_8$ for $d=8$ and 
the Leech lattice for $d=24$. 

In the particular case of $d=2$ the conjecture reduces to
\[
\mathcal{E}(s,n)=V_s(\S^2) n^2+ 
 \frac{(\sqrt{3}/2)^{s/2} \zeta_{\Lambda_2}(s)}{(4\pi)^{s/2}} n^{1+s/2} 
+o(n^{1+s/2}),
\]
where $\zeta_{\Lambda_2}(s)$ is the zeta function of the hexagonal lattice. 
This zeta function can be evaluated by using its relation with a 
particular Dirichlet L-series, see \cite{BHS12}.

When $n$ is of the form $\pi_L,$ since $\mathcal{E}(s,n)\le \mathbb{E} E_s(n)$, 
we get from Theorem~\ref{th:asymptoticrieszsharmonic} an upper bound for the 
minimal energy, for $d\ge 2$ and $0<s<d$:
\[
\mathcal{E}(s,n)-V_s(\S^d) n^2 \le -C_{s,d} n^{1+s/d}+o(n^{1+s/d}),
\] 
for $C_{s,d}$ as in \eqref{constanttheorem2}.

% Clearly as $\mathcal{E}_s(n)\le \mathbb{E} E_s(n)$, we have that for $d\ge 2$ 
% and $0<s<d$ the constant $c>0$ in \eqref{knownboundsenergy}
% is such that
% \[c\ge C_{s,d}=2^{s-s/d}V_s(\S^d)\left(d!\right)^{-1+\frac{s}{d}}
% \frac{ d\,\Gamma\left( 1+\frac{d}{2} \right)   \Gamma\left( \frac{1+s}{2} 
%\right)\Gamma\left(d-\frac{s}{2}\right) }
% { \sqrt{\pi} \Gamma\left( 1+\frac{s}{2} \right) \Gamma\left( 1+\frac{s+d}{2} 
%\right)}\]

For $d=2$ this bound is a bit worse ($C_{s,2}$ is smaller) than the lower 
constant $2^{-s}\Gamma(1-\frac{s}{2})$ from \cite[Corollary 1.4]{ASZ14} given 
by the spherical ensemble, see figure \ref{fig:constantesd2}.

\begin{figure}
    \begin{center}
           \includegraphics[width=0.9\textwidth]{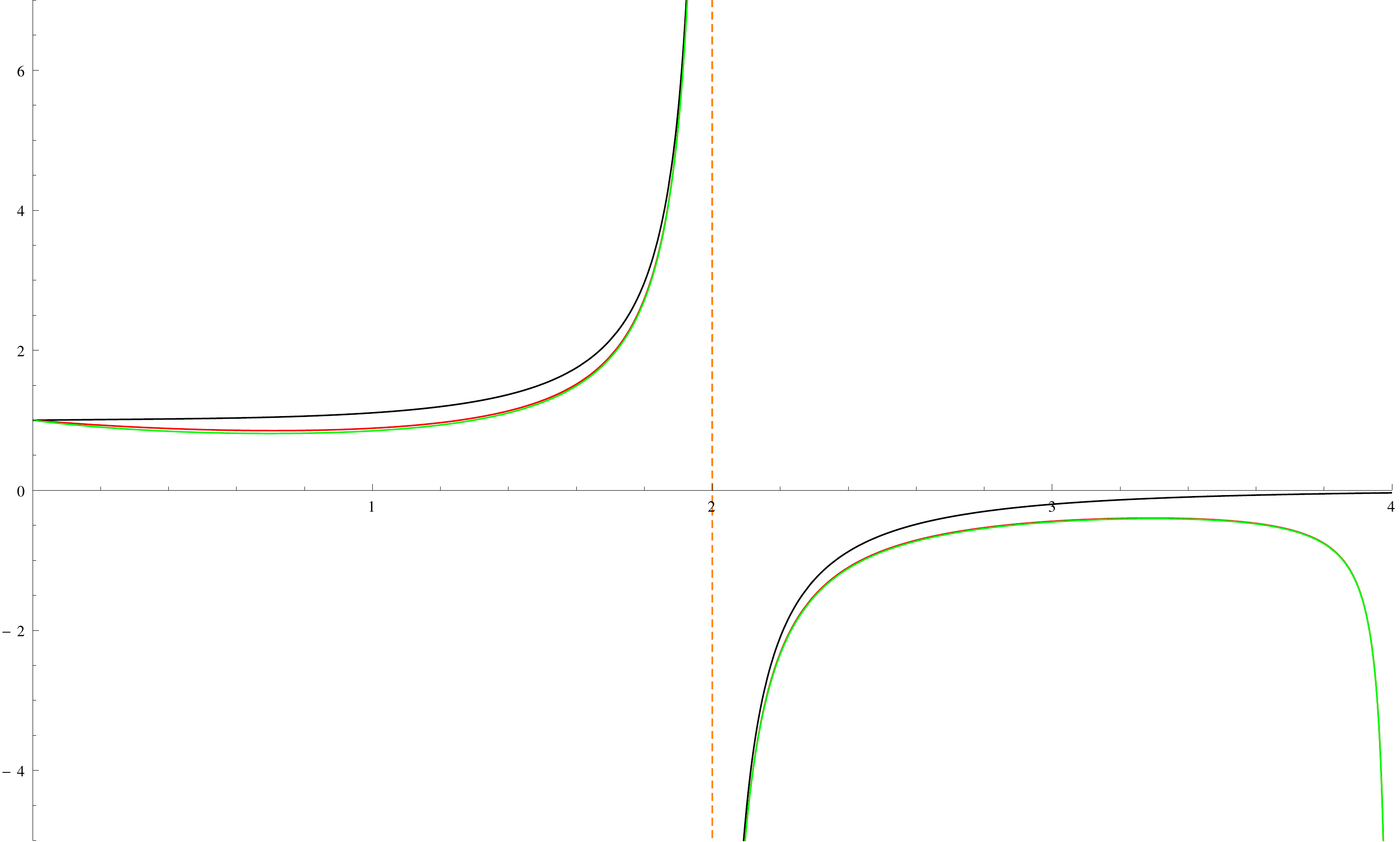}
          \caption{Graphic of $-\frac{(\sqrt{3}/2)^{s/2} 
\zeta_{\Lambda_2}(s)}{(4\pi)^{s/2}}$ in black, $2^{-s}\Gamma(1-\frac{s}{2})$
           in red and the constant $C_{s,2}$ of \eqref{constanttheorem2} in 
green (case $d=2$). Note that Theorem \ref{th:asymptoticrieszsharmonic} only 
involves this constant for $s<2$. We plot the constants in a larger interval for 
illustration purposes.}\label{fig:constantesd2}
    \end{center}
\end{figure}

For larger $d,$ according to \cite[Theorem 3.8.2]{Leopardi_thesis} and 
\cite{RSZ94}, the known upper bound, $C$ in \eqref{knownboundsenergy},
equals $1/Q^s$ where $Q$ is the minimal constant such that
one can construct an area-regular partition $\{D_j \}$ of $\S^d$ 
with diameter $\operatorname{diam}(D_j)\le Q n^{-1/d}.$ 
Observe that the area of the spherical cap of (small) radius $r$ is essentially 
equal to
$\frac{r^d}{d}\omega_{d-1}$ and therefore the radius of a spherical cap of 
area 
$\omega_d/n$ is approximately
$(d \omega_d/\omega_{d-1})^{1/d} n^{-1/d}$. This implies, together with the 
fact 
that the spherical cap has the smallest diameter among the 
sets with the same area, that $Q\ge 2(d \omega_d/\omega_{d-1})^{1/d}$. In fact, 
it is known how to construct 
area-regular partitions with diameter $8(d \omega_d/\omega_{d-1})^{1/d}$ (for 
$d=2$ one can get better constants but always with $Q\ge 4$). These constants
are worse than the constant given by Theorem \ref{th:asymptoticrieszsharmonic}.

It has been recently shown, \cite{PS2015}, that in the case of $\R^d$ and 
$d-2<s<d$ the second term of the minimal energy has indeed the form $B_{s,d} 
n^{1+s/d}$ for a constant $B_{s,d}.$

In the range $d-1<s<d$ we get the following result.

\begin{corollary}\label{cor:allpoints1}
  For any $n\ge 1$ (not necessarily of the form $\pi_L$) for $d-1<s<d$ we have 
that
\[
  \mathcal{E}(s,n)\le V_s(\S^d) n^2 -C_{s,d} 
n^{1+s/d}+o(n^{1+s/d}),
 \]
where $C_{s,d}$ is the constant in \eqref{constanttheorem2}.
\end{corollary}

Indeed, this Corollary follows easily from the fact that for $n\in 
(\pi_L,\pi_{L+1})$ 
\[\mathcal{E}(s,n)\le \mathcal{E}(s,\pi_{L+1})\le 
\mathbb{E}_{x\in(\S^d)^{\pi_{L+1}}}(E_s(x)),\]
and both $\mathbb{E}_{x\in(\S^d)^{\pi_L}} (E_s(x))$ and $\mathbb{E}_{x\in(\S^d)^{\pi_{L+1}}}(E_s(x))$
have the same two first asymptotic terms. More precisely, if
$\pi_L=A_d L^d+O(L^{d-1}),$ also $\pi_{L+1}=A_d L^d+O(L^{d-1}),$ and
\begin{align*}
 \mathbb{E}_{x\in(\S^d)^{\pi_{L+1}}}(E_s(x)) & =V_s(\S^d) A^2_d L^{2d}+O(L^{2d-1})-C_{s,d} 
A_d^{1+s/d}L^{d+s}+o(L^{d+s})
 \\
 &
=V_s(\S^d) A^2_d L^{2d}-C_{s,d} A_d^{1+s/d}L^{d+s}+o(L^{d+s})
\\
 &
=V_s(\S^d) n^2-C_{s,d} n^{1+s/d}+o(n^{1+s/d}),
\end{align*}
when $d-1<s$ and $n\in (\pi_L,\pi_{L+1}).$ 

For the logarithmic potential the continuous energy is
\[
V_{\rm log}(\S^d)=\int_{\S^d} \int_{\S^d} \log\frac{1}{\|x-y\|} 
d\mu(x)\,d\mu(y)=\frac{1}{2}(\psi_0(d)-\psi_0(d/2))-\log 2,
\]
where $\psi_0=(\log\Gamma)'$ is the digamma function.
Note that  we have
\begin{align}\label{eq:VlogVs}
 V_{\rm log}(\S^d)=\frac{d}{ds}\Bigr|_{s=0}\left(V_s(\S^d)\right).
\end{align}
In the 
computation of the derivative of generalized hypergeometric function in Theorem 
\ref{th:rieszsharmonic} most of the terms 
vanish, and we get a closed expression for the expected energy.

\begin{theorem}\label{th:logharmonic}
 Let $x=(x_1,\ldots,x_n)\in (\S^d)^n,$ where $n=\pi_L,$ be $n$ points drawn from the harmonic ensemble. 
Then,
% \begin{align*}
% \mathbb{E} E_0(x)=&\pi_L^2 V_{\mathrm log}(\S^d)\\
% &-\frac{\pi_L}{2}
% \left(\sum_{k=1}^L 
% \frac{1}{\frac{d}{2}+k}+H_{L+d-1}+\psi_0(1/2)-\psi_0(d/2)\right), 
% \end{align*}
\[
\mathbb{E}_{x\in (\S^d)^n}( E_0(x))=n^2 V_{\rm log}(\S^d) -\frac{n}{2}
\left(\sum_{k=1}^L 
\frac{1}{\frac{d}{2}+k}+H_{L+d-1}+\psi_0(1/2)-\psi_0(d/2)\right),  
\]
where $H_k=\sum_{j=1}^k \frac{1}{j}$ stands for the $j$th harmonic number.
\end{theorem}

\begin{corollary}\label{cor:logharmonic}
 Let $x=(x_1,\ldots,x_n)\in (\S^d)^n,$ where $n=\pi_L,$ be $n$ points drawn from the harmonic ensemble. 
Then,
\[
\mathbb{E}_{x\in (\S^d)^n}( E_0(x))
=n^2 V_{\rm log}(\S^d)
-\frac{1}{d}n \log n +C_d n+o(n),
\]
where
\[
 C_d=\frac{1}{d}\log \frac{2}{d!}+\log 2+\psi_0\left( \frac{d}{2} 
\right)+\frac{1}{d}.
\]
In particular, for $d=2$, we have $C_2=1/2+\log 2-\gamma\approx0.6159.$.., 
where $\gamma$ is the Euler-Mascheroni constant.
\end{corollary}

For the logarithmic case it is known that
\[
\mathcal{E}(0,n)=V_{\rm log}(\S^d) n^2 -\frac{1}{d}n\log n+O(n).
\]
There is a conjecture about the value of the asymptotic coefficient of $n$ for 
$d=2,4,8$ and $24$ which is similar to the one above, see \cite{BHS12}. For 
$d\ge 3$ there is a lower bound for this coefficient, see \cite{Bra08}, and it is 
negative. But for $d\ge 3$, it is not known if the limit $\lim_{n\to \infty} 
\left[\mathcal{E}_{log}(n)-V_{\rm log}(\S^d) n^2+\frac{1}{d}n\log 
n\right]/n$ exists. For $d=2$, B\'etermin and Sandier \cite{Bet14} show that the 
corresponding limit exists and that the conjectured value
\[
\frac{1}{d}\log 
\frac{\omega_d}{|\Lambda_d|}+\zeta'_{\Lambda_d}(0)=-0.0556...,
\]
would be correct if the triangular lattice was a minimizer of the Coulombian 
renormalized energy introduced by Sandier and Serfaty, 
\cite{SanSer2015}.

As before, we can easily get upper bounds for the constant involved in 
the above conjectures but, as the constant seems to be negative (it is for 
$d=2$), we think these bounds are not so interesting.

But for the sake of illustration:

  When $d=2$ we get $C_2=\log 2+\frac{1}{2}-\gamma\approx 0.6159...$ while the 
constant it is known to lie in $(-0.2254,-0.0556)$, see \cite{Bet14}. By taking 
the expected logarithmic energy of the point process in $\S^2$ given by the 
zeros of polynomials with random coefficients via the stereographic projection, 
the authors in \cite{ABS11} get the value $\log 2-\frac{1}{2}= 0.1931...$. By 
using the determinantal process with exponential decay introduced by Krishnapur, 
i.e. the spherical ensemble, the authors in \cite{ASZ14} get $\log 
2-\frac{\gamma}{2}\approx 0.4045...$.

A final remark is in order: in the case $d=1$ we have that $V_{\rm  
log}(\S^1)=0$ and
\[
\mathbb{E}_{x\in(\S^1)^n} (E_0(x))=-n \log n +(1-\gamma) n+o(n),
\]
while the minimal energy it is known to be $\mathcal{E}(0,n)=-n\log n$ (the energy of the roots of unity).  One 
cannot improve the constant $1-\gamma\approx 0.4227...$ by taking other 
rotation invariant projection kernels because there is just one. As we will see 
in the next section in several dimensions there are more kernels.

  In the limiting case $s=d$ the optimal continuous energy is not finite and 
this case is called singular. In the discrete setting it is known from 
\cite{KS98} that 
\[
\lim_{n\to \infty} \frac{\mathcal{E}(d,n)}{n^2 \log n}=\frac{\omega_{d-1}}{d 
\omega_d}.
\]

It was shown in \cite[Proposition 2]{BHS12} that
\[
-c(d)n^2 +O(n^{2-2/d}\log n) \le \mathcal{E}(d,n)- \frac{\omega_{d-1}}{d 
\omega_d}n^2 \log n\le \frac{\omega_{d-1}}{d \omega_d}n^2 \log \log n+o(n^2),
\]
with 
\[
c(d)=\frac{\omega_{d-1}}{d \omega_d}\left( 1-\log \frac{\omega_{d-1}}{d 
\omega_d}+d[\psi_0(d/2)-\psi(1)-\log 2]  \right).
\]

And it was conjectured \cite[Conjecture 5]{BHS12} that

\[
\mathcal{E}(d,n)=\frac{\omega_{d-1}}{d \omega_d}n^2 \log n+A_{d,d} 
n^2+O(1),
\]
where
\[
A_{d,d}=\lim_{s\to d}\left[V_s(\S^d)+  \frac{A_{s,d}}{\omega_d^{s/d}} 
\right],
\]
and $A_{s,d}$ is the constant in \eqref{constantconjectured}.
Observe that when $d=2$ we have $A_{2,2}=-0.0857...$.

In the case $d=2$ it was shown in \cite{ASZ14} that the correct order of the 
second term 
is indeed $n^2$ by showing that the expected $2$-energy of the spherical ensemble is
\[
\frac{1}{4} n ^2 \log n+ \frac{\gamma}{4}n^2-\frac{n}{8}-\frac{1}{48}+O(n^{-2}).
\]
We get a similar result.

\begin{theorem}\label{th:rieszsingular}
 Let $x=(x_1,\ldots,x_n)\in (\S^d)^n,$ where $n=\pi_L,$ be $n$ points drawn from the harmonic ensemble. 
Then
 \[
\mathbb{E}_{x\in(\S^d)^n}(E_d(x))=\frac{\omega_{d-1}}{d \omega_d}n^2 \log 
n+C_{d,d} n^2 +o(n^2),
\]
where
\begin{equation}		\label{constantlogarithm}
C_{d,d}=\frac{\omega_{d-1}}{2 
\omega_d}\left(\psi_0(d+1)-\psi_0\left(\frac{d}{2}
+1\right)\right)-\psi_0\left(\frac{d}{2}\right)-\frac{1}{d}-\frac{1}{d}\log 
\frac{2}{d!}.  
\end{equation}

\end{theorem}

For example, when $d=2$ we get 
\[
\mathbb{E}_{x\in(\S^2)^n}(E_2(x))=\frac{1}{4} n ^2 \log n+ \left( 
\gamma-\frac{3}{8}\right) n^2+o(n^2),
\]
so a larger energy than in \cite{ASZ14} as $\gamma-\frac{3}{8}=0.2022...$ 
and $\frac{\gamma}{4}=0.1443...$.
  
As in Corollary \ref{cor:allpoints1}, we get a bound for 
$\mathcal{E}(d,n)$ for all $n$ (not only of the form $\pi_L$) and we get the 
correct order
for the second asymptotic term.

\begin{corollary}\label{cor:allpoints2}
  For any $n\ge 1$ (not necessarily of the form $n=\pi_L$),
\[
  \mathcal{E}_d(n)\le \frac{\omega_{d-1}}{d \omega_d}n^2 \log n+C_{d,d} n^2 
+o(n^2).
 \]
where $C_{d,d}$ is the constant in \eqref{constantlogarithm}.
\end{corollary}

Indeed, see discussion after Corollary \ref{cor:allpoints1} and observe that 
the computation works when $s=d.$

\subsection{Optimality for isotropic projection  kernels}  

In our next results we deal with more general kernels. We assume that our 
kernel is invariant by rotations i.e.
\[
 \mbox{when}\;\;d(x,y)=d(z,t)\;\;\mbox{then}\;\; K(x,y)=K(z,t), \;\;\;\;x,y,z,t\in \S^d.
\]
This implies that the random point field is invariant by rotations, or 
isotropic, and that it can be written as $K(\langle x,y\rangle)$ for some 
$K:[-1,1]\rightarrow \C$. If we want that this kernel generates a determinantal 
process, the function $K$ should be positive definite in $\S^d$ and by 
Schoenberg's Theorem we get that, see \cite{Sch42} or \cite[Th.~1, 
p.~123]{CheneyLight}, it has the form 
\begin{equation}\label{eq:expansionesferas}
 K(x,y)=K(\langle x,y\rangle),\quad K(t)=\sum_{k=0}^\infty a_k C_k^{\frac{d-1}{2}}(t),
\end{equation}
where $C_k^{\frac{d-1}{2}}$ is the Gegenbauer polynomial 
% normalized in such a way that
% \begin{equation}\label{eq:Ck1}
%  C_k^{d/2-1/2}(1)=\binom{d+k-2}{k},
% \end{equation}
and the $a_k\ge 0 $ satisfy:
\[
 \mbox{trace}(K)=K(1)=\sum_{k=0}^\infty a_k\binom{d+k-2}{k}<\infty.
\]
From Macchi-Soshknikov theorem 
\cite[Theorem 4.5.5]{GAF} the fact that $0\le a_k \le \frac{2k+d-1}{d-1}$ is needed also to get a 
determinantal process as the operator has to have spectrum in $[0,1]$ and
$\frac{2k+d-1}{d-1}C_k^{\frac{d-1}{2}}(\langle x,y \rangle)$ (the zonal harmonic of 
degree $k$) is 
the projection kernel onto $\mathcal{H}_k$.

As we want the process to have $n$ points a.s. and it is known that the total 
number of 
points in the process has the distribution of a sum 
of independent Bernoulli's with parameters $a_k\frac{d-1}{2k+d-1},$ see 
\cite[Theorem 4.5.3]{GAF}, we impose that our kernel is a projection kernel
and therefore
\begin{equation}\label{eq:form}
a_k=\begin{cases}
  \frac{2k+d-1}{d-1} &\text{for finitely many $k$,}\\
  0&\text{otherwise,}
  \end{cases}
\end{equation}
with 
\begin{equation}			\label{traceequaln}
\sum_{k=0}^\infty a_k\binom{d+k-2}{k}=n. 
\end{equation}
Note that such a sequence $a_k$ does not exist for all values of $n$.  It does 
however exist for an infinite sequence of $n$ (which depends on $d$) including 
those $n$ of the form $n=\pi_L$ for some positive integer $L$.

% There are different ways to choose the $a_k$ satisfying these properties. For 
% example, if $d=2$ and $n=(m+1)^2$ is a perfect square then choosing 
% $a_k=2k+1$ if $k\leq m$ and $a_k=0$ if $k>m$ does the work.

In such a general setup we get an expression in terms of the 
integral of the kernel.

\begin{theorem}\label{th:rieszsgeneral1}
 Let $x=(x_1,\ldots,x_n)\in (\S^d)^n$ be $n$ points generated by the determinantal random point 
process associated to the kernel $K$. Then, for $0<s<d$,
 \[
  \mathbb{E}_{x\in(\S^d)^n}(E_s(x))= 
\frac{\omega_{d-1}}{\omega_d2^{s/2}}\Bigl(n^2 2^{d-1-\frac{s}{2}} 
\operatorname{B}\Bigl(\frac{d}{2},\frac{d}{2}-\frac{s}{2}\Bigr)-\int_{-1}
^1\frac{
|K(t)|^2(1-t^2)^{d/2-1}}{(1-t)^{s/2}}\,dt\Bigr).
 \]
\end{theorem}

In the particular case of the Riesz $2$-energy and $\S^d$ for $d\ge 3$ one can 
get, after lengthy computations, the following explicit expression for the 
energy in terms of the coefficients of the kernel. We haven't been able to get 
simple expressions like this one for other energies.

\begin{theorem}\label{th:case2}
 In the setting of Theorem \ref{th:rieszsgeneral1}, for  $s=2$ and $d\geq3$ we 
have
 \[
   \mathbb{E}_{x\in(\Sd)^n}(E_2(x))=V_2(\S^d)\left(n^2 -\sum_{\ell=0}^\infty 
a_{\ell}\binom{d+\ell-2}{\ell}\left(a_\ell+2\sum_{j>\ell} a_j \right)\right).
 \] 
\end{theorem}

The following result provides a criterion to compare energies given by different 
kernels. In particular it implies that the harmonic kernel
gives the smallest expected $2$-energy among the different isotropic kernels 
with the same trace.

\begin{theorem}\label{th:optimal}
 Let $K_a$ and $K_b$ be two kernels with coefficients $a=(a_0,a_1,\ldots)$ and 
$b=(b_0,b_1,\ldots)$ satifying conditions 
\eqref{eq:form}, \eqref{traceequaln}. Let $\mathbb{E}_a$ and $\mathbb{E}_b$ 
denote respectively 
the expected value of $E_2(x)$ when $x$ is given by the determinantal point 
process associated to $K_a$ and $K_b.$ Assume that for every $i,j\in\N$ we have:
 \begin{equation}\label{eq:if}
  \text{if }i<j,a_i=0\text{ and }a_j>0\text{ then }b_i=0.
 \end{equation}
 Then, $\mathbb{E}_a\leq\mathbb{E}_b$, with strict inequality unless $a=b$. In 
particular, the harmonic kernel is optimal since \eqref{eq:if} is trivially 
satisfied in that case. 
\end{theorem}

\begin{remark}
 Note that the hypotheses \eqref{eq:if} just means that, if there are ``holes'' 
(i.e. intermediate zeros) in the sequence $a_0,a_1,\ldots$ then the sequence 
$b_0,b_1,\ldots$ must also have these holes (thus, informally, 
Theorem~\ref{th:optimal} means that more holes imply larger energy).
\end{remark}

It is natural to ask if the optimality of the harmonic kernel remains true for 
general $s$. We thus propose the following conjecture.
\begin{conjecture}
 The harmonic kernel is optimal for all $s\geq0$ in the sense that if $K$ is 
another isotropic kernel producing $n=\pi_L$ points in $\S^d$ then the expected 
value of $E_s(x)$ when $x$ is drawn from the point process given by $K$, is 
larger than the expected value of $E_s(x)$ when $x$ is drawn from the harmonic 
ensemble. 
\end{conjecture}

In \cite[Theorem 4.4]{MNPR15} the authors show that among isotropic kernels the reproducing 
kernel is optimal with respect to some measures of repulsiveness defined in terms of the second 
intensity function (as in our case).

\subsection{Expectation and Variance of linear statistics}

Another measure of the uniformity of the distribution of the harmonic 
ensemble is the computation of the variance of linear statistics.

Let $\mathcal{X}$ the point process with $\pi_L$ points a.s. in $\S^d$ defined 
by the harmonic ensemble. 
We denote by $\mu_L$ the 
empirical measure associated to a realization $x_1,\ldots, x_{\pi_L}\in \S^d$ 
of this process
i.e. 
\[
\mu_L= \sum_{1\le j\le \pi_L}\delta_{x_j}.
\]
Given a function $\phi$ on the sphere,
we denote by $\mathcal{X}(\phi)$ the linear statistic associated to $\phi$, 
i.e. the 
random variable
\[
 \mathcal{X}(\phi) = \int_{\S^d} \phi d\mu_L= \sum_{1\le j\le \pi_L} 
\phi(x_{j}).
\]
The expected value of $\mathcal{X}(\phi)$ is easily computed with the first 
intensity function of 
the harmonic ensemble:
\[
 \mathbb E (\mathcal{X}(\phi)) = \pi_L  \int_{\S^d} \phi d\mu \sim L^d.
\]

When $\phi$ is the characteristic function of a spherical cap $\phi=\chi_A,$ 
the random variable $\mathcal{X}(\chi_A)$ is the number of points in 
$A$ that we denote as $n_A.$ For this case of a rough linear statistic we get the following result.

\begin{proposition}\label{prop:roughvariance}
Let $A=A_L$ be a spherical cap of radius $\theta_L\in [0,\pi)$ with 
\[
\lim_{L\to \infty} \theta_L\in [0,\pi),
\]
and  $L \theta_L\to \infty$ when $L\to \infty$. Let $n_A$ be the number of points
in $A$ among $\pi_L$ points drawn from the harmonic ensemble. Then
\[
\operatorname{Var}(n_A)\lesssim L^{d-1}\log L+O(L^{d-1}),
\]
where the constant is $\lim_{L \to \infty}{\theta_L}^{d-1} \frac{4}{2^{d} \pi  
\Gamma\left(\frac{d}{2} \right)^2 }$.
\end{proposition}

From the proposition above one can deduce, following Beck 
\cite[Theorem 2]{Bec84} or 
\cite[Theorem 1.1.]{ASZ14} for this determinantal setting, the following result 
about the spherical cap discrepancy.

\begin{corollary} 
  For every $M>0,$ the spherical cap discrepancy of a set of $n=\pi_L$ points 
$x=(x_1,\dots , x_n)\in (\S^d)^n$ drawn from the harmonic ensemble 
  satisfies
\[
\mathbb{D}(x)=O(L^{-\frac{d+1}{2}}\log L)=
O(n^{-\frac{1}{2}(1+\frac{1}{d})}\log n),
\]
with probability $1-\frac{1}{n^M}$, i.e. with overwhelming probability.
\end{corollary}

When $\phi$ is a smooth function on the sphere we have a better result.

\begin{proposition}\label{prop:smooth}
 Let $\phi\in \mathcal C^1(\S^d)$ then
 \[
  \operatorname{Var}(\mathcal{X}(\phi))\lesssim L^{d-1}.
 \]
 Moreover this cannot be improved in general. If $\phi(x)=x_i$  where $x_i$ is 
any coordinate function then 
 \[
    \operatorname{Var}(\mathcal{X}(\phi))\sim L^{d-1}.
 \]
\end{proposition}

Thus in the harmonic setting there is a gain of a $\log L$ term when we move 
from rough linear statistics to smooth linear statistics. 

This is in contrast with the spherical ensemble setting in $S^2,$ for which one can see that the improvement is much better. 
Indeed, all we need to use is that the kernel associated to the spherical ensemble, denoted as $S_n(x,y),$
satisfies the  estimate
\[
 |S_n(x,y)|\lesssim n \exp (-C n d^2(x,y)),
\]
see \cite[Formula~(4.1)]{ASZ14}. For the rough case, the variance of the number of points in a spherical cap of constant radius, for $n$ points drawn from
the spherical ensemble,
was computed in \cite[Lemma~2.1]{ASZ14} and it is of 
order $n^{1/2}.$ In the smooth case, if we take a Lipschitz linear statistic with symbol $\phi$ and $\mathcal{X}$ is the point process 
with $n$ points defined by the spherical ensemble we have that
\[
 \operatorname{Var}(\mathcal{X}(\phi))\lesssim 
 \int_{\S^2\times 
\S^2}\left|S_{n}(x,y)\right|^{2}d^2(x,y)d\mu(x)d\mu(y)
\]
\[
\lesssim 
\int_{\S^2} n^2 d^2(x,\mathbf{n})\exp (-C n d^2(x,\mathbf{n}))d\mu(x) \le C,
\]
where $\mathbf{n}=(0,0,1)\in \S^2$ stands for the north pole.
Therefore, in the spherical ensemble, there is a gain of a power 
$n^{1/2}$ when we consider smooth statistics instead of rough statistics.

All this information has been condensed in Table \ref{table:variance}.
For reference it is also included the variance of linear statistics of the 
point process generated by random elliptic polynomials which was studied in 
\cite{ST04}. To get a fair comparison between processes it is convenient to 
consider the case of $n=L^2$ points in the spherical ensemble and in the random 
polynomial case.

\begin{center}
\begin{table}
 \caption{Expectation and Variance of linear statistics with different 
rotation--invariant random 
point processes in the sphere}\label{table:variance}
\begin{tabular}{|l|c|c|c|c|}\hline
 &Harmonic & Spherical  &Zeros of random   \\ 
 &ensemble & ensemble  & polynomials \\
&($n\sim L^d$)&with $n$ points&of degree $n$ \\
&&($d=2$)&($d=2$)\\
\hline
Expectation & $L^d$ &$n$ & $n$\\ \hline
Var. Rough & $L^{d-1}\log L$ &$n^{1/2}$&  $n^{1/2}$\\ \hline
Var. Smooth & $L^{d-1}$ &$n^0$& $n^{-1}$\\ \hline
\end{tabular}
\end{table}
\end{center}

\subsection{Separation distance}
Following \cite{ASZ14} we also consider the probability distribution of the 
separation distance $\sepdist(x)$ of a given $x\in(\S^d)^n$, and its counting 
version
\[
 G(t,x)=\sharp\{(i,j):i< j,\|x_i-x_j\|\leq t\}.
\]
We have the following result which gives a sharp bound on the expected value of 
$G(t,x)$ for the harmonic kernel.
\begin{proposition}\label{prop:G}
Let $x=(x_1,\ldots,x_n)\in (\S^d)^n$ be $n=\pi_L$ points drawn from the harmonic ensemble. 
Then, for
\[
 t\leq\frac{d+6}{(2L+d)L},
\]
we have
\[
 \mathbb{E}_{x\in (\S^d)^n}(G(t,x))\leq 
\frac{L(L+d)\pi_L^2\omega_{d-1}}{2(d+2)^2\omega_d}t^{d+2}=C_d n^{2+2/d}t^{
d+2}+o(n^{2+2/d})t^{d+2},
\]
where
\[ 
C_d=\frac{\Gamma(d+1)^{2/d}\omega_{d-1}}{2^{1+2/d}(d+2)^2\omega_d}.
%\stackrel{d>>1}{\approx}\frac{\sqrt{d}}{2^{3/2}e^2\pi^{1/2}}.
\]
\end{proposition}

Observe that $C_d$ above is asymptotically (for large $d$) as $\frac{\sqrt{d}}{2^{3/2}e^2\pi^{1/2}}.$
Note that $\sepdist(x)\leq t$ implies $G(t,x)\geq 1$, hence 
$\mathbb{P}(\sepdist(x)\leq t)\leq \mathbb{P}(G(t,x)\geq1)\leq 
\mathbb{E}(G(t,x))$. We thus have:
\begin{corollary}\label{cor:sepdist}
Let $x=(x_1,\ldots,x_n)\in (\S^d)^n$ be $n=\pi_L$ points drawn from the harmonic ensemble. 
For 
$\alpha\in(0,\frac{d+6}{(2L+d)L})$ we have
 \[
  \mathbb{P}\left(\sepdist(x)\leq\alpha n^{-\frac{2d+2}{d^2+2d}}\right)\leq  
\mathbb{E}_{x\in (\S^d)^n}(G(\alpha n^{-\frac{2d+2}{d^2+2d}},x))\leq C_d \alpha^{d+2}+o(1).
 \]
\end{corollary}
From Corollary \ref{cor:sepdist}, an 
$n$-tuple $(x_1,\ldots,x_n)\in (\S^d)^n$ drawn from the harmonic ensemble likely satisfies
\[ 
\sepdist(x)\geq \Omega\left(n^{-\frac{2d+2}{d^2+2d}}\right). 
\]
If each point is generated randomly and uniformly in the sphere $\S^d$ it is 
easy to see that one can just expect $\sepdist(x)\geq \Omega(n^{-2/d})$ which is 
a worse estimate since
\[
 \frac{2d+2}{d^2+2d}<\frac2d,\quad d\geq 1.
\]
In the two-dimensional case $,d=2,$ the spherical ensemble studied in \cite{ASZ14} satisfies 
$\sepdist(x)\geq \Omega(n^{-3/4})$ for large $n$ (see \cite[Corollary~1.6]{ASZ14}). 
For the harmonic ensemble we have also 
$\sepdist(x)\geq \Omega(n^{-3/4})$ for $d=2.$ Moreover, the expected value of 
$G(\alpha/n^{3/4},x)$ in both cases satisfies (assymptotically for 
$n\rightarrow\infty$)
 \[
  \mathbb{E}_{x\in (\S^d)^n}\left(G\left(\frac{\alpha}{n^{3/4}},x\right)\right)\leq 
\frac{\alpha^{4}}{64}.
 \]
However, as pointed out before, the optimal separation 
distance is of order $n^{-1/d}.$ Therefore, the separation distance of the harmonic ensemble 
is better (larger) than the one obtained by uniform points in $\S^d,$ but it is still far from the 
optimal value.

%===============================================================================
%===============================================================================
%===============================================================================

\section{Proofs}

\subsection{Riesz \texorpdfstring{$s$}{s}-energy. From Theorem~\ref{th:rieszsharmonic} to
Theorem~\ref{th:rieszsingular}} \label{eq:Rieszs}

We start with a lemma about integration of zonal functions.

\begin{lemma}\label{lem:computeintegralsphere}
 Let $v\in\S^d$ and let $f:\S^d\rightarrow[0,\infty)$ be such 
that $f(u)=g(\langle u,v\rangle)$ for some measurable function $g$ defined in $[-1,1].$
Then
 \[
  \int_{\S^d}f(u)\,d \mu(u)=\frac{\omega_{d-1}}{\omega_{d}} \int_{-1}^1 
g(t)(1-t^2)^{d/2-1}\,dt.
 \]
\end{lemma}
\begin{proof}
 This is a particular case of Funck-Hecke formula, see 
\cite[Theorem~6]{Muller66}. It can also be proved directly using the change of variables theorem for the projection parallel to the space $v^\perp$ defined from the sphere to the cylinder.
\end{proof}

We thus have:
\begin{lemma}\label{lem:integral2}
Let $\mathbf{n}=(0,\ldots,0,1)\in \S^d$ be the north pole. For $0<s<d$,
\[
  \int_{x\in\S^d}\frac{P_L^{(1+\lambda,\lambda)}(\langle 
x,\mathbf{n}\rangle)^2}{\|x-\mathbf{n}\|^s}\,d\mu (x)= 
\frac{\omega_{d-1}}{2^{s/2}\omega_d}\int_{-1}^1 P_L^{(1+\lambda,\lambda)}(t)^2 
(1-t)^{\lambda-\frac{s}{2}}(1+t)^{\lambda}\,dt=
\]
\[ 
\frac{2^{d-1-s}\omega_{d-1} C_{s,d}(L) \Gamma\left( \frac{d-s}{2} \right) }
{ \omega_d\Gamma\left( 1+\frac{d}{2} \right)\Gamma\left( 1+\frac{s}{2} \right)} 
\\
\cdot {}_{4}F_{3}\bigl( 
-L,d+L,\frac{d-s}{2},-\frac{s}{2};\frac{d}{2}+1,d-\frac{s}{2}+L,-\frac{s}{2}
-L;1  \bigr).
\]
\end{lemma}
\begin{proof}
The first equality follows directly from Lemma~\ref{lem:computeintegralsphere}. 
The value of that integral can be found in standard integral 
tables, see 
for example 
\cite[p. 288]{EOT54}. This finishes the proof of the lemma.
\end{proof}

% \begin{remark}
%  The special case $s=2$ admits an alternative expression, see 
% Corollary~\ref{cor:s2harmonicassymptotics} below.
% \end{remark}
%===============================================================================
\begin{proof}[Proof of Theorem~\ref{th:rieszsharmonic}] From 
Corollary~\ref{cor:rieszenergies} we have
\begin{align}
\mathbb{E}_{x\in(\S^d)^n}(E_s(x))=&\pi_L^2\int_{u,v\in\S^d}\frac{1-\frac{1}{
\binom{L+\frac{d}{2}}{L}^2}P_L^{(1+\lambda,\lambda)}(\langle 
u,v\rangle)^2}{\|u-v\|^s}\,d\mu (u)\,d\mu (v) \nonumber\\
=&\pi_L^2\int_{u\in\S^d}\frac{1} 
{\|u-\mathbf{n}\|^s}-\frac{P_L^{(1+\lambda,\lambda)}
(\langle u,\mathbf{n}\rangle)^2} 
{\binom{L+\frac{d}{2}}{L}^2\|u-\mathbf{n}\|^s}\,d\mu 
(u),\label{eq:intermedia}
\end{align}
where $\mathbf{n}=(0,\ldots,0,1)\in \S^d$ is the north pole. The last equality 
follows from the rotation invariance of the functions involved. Now we see from 
\eqref{eq:intermedia} that the expected value of the energy $E_s(x)$ is finite if and only if 
$s<d+2$. Indeed, sending a small cap around the north pole $\mathbf{n}$ to 
$\R^d$ through the projection onto the first $d$ coordinates and using some 
trivial bounds, the integral in \eqref{eq:intermedia} is finite if and only if 
for some $\epsilon>0$ we have
 \begin{equation}\label{eq:isfinite}
  \int_{x\in\R^{d},\|x\|<\epsilon}\frac{1}{\|x\|^s} 
\left(1-\frac{P_L^{(1+\lambda,\lambda)}
(1-\|x\|^2/2)^2}{\binom{L+\frac{d}{2}}{L}^2}\right)\,d x<\infty.
 \end{equation}
 An elementary bound on the value of Jacobi polynomials is shown in Lemma \ref{lemma:jacobibound}. Using that result we have
 \[
  1-\frac{P_L^{(1+\lambda,\lambda)}
(1-\|x\|^2/2)}{\binom{L+\frac{d}{2}}{L}}=\frac{L(L+d)}{2d+4}\|x\|^2+O(\|x\|^4).
 \]
We thus have
\[
 \frac{1}{\|x\|^s}\left(1-\frac{P_L^{(1+\lambda,\lambda)}
(1-\|x\|^2/2)^2}{\binom{L+\frac{d}{2}}{L}^2}\right)\sim \|x\|^{2-s},
\]
and passing to polar coordinates the integral in \eqref{eq:isfinite} is finite 
if and only if $s<d+2$.

Now we center in the case $0<s<d$. From the definition of $V_s(\S^d),$ see \eqref{continuousrieszenergy}, and 
the invariance under 
rotations of $\|\cdot\|$ we get that
\[
    \int_{\S^d}\frac{1}{\|x-\mathbf{n}\|^s}\,d\mu (x)=V_s(\S^d).
\]
For the second summand in (\ref{eq:intermedia}) we use Lemma \ref{lem:integral2}.
This, together with the asymptotic formula for the gamma function, to get the 
asymptotic behavior $C_{s,d}(L)\sim L^s$,  finishes the proof of 
Theorem~\ref{th:rieszsharmonic}.  
\end{proof}

Our goal is to study the first terms of the asymptotic expansion of the energy. 
To some extent, it is possible to get information from the generalized 
hypergeometric function but we will use estimates of the Jacobi polynomials to 
get a complete answer.

\begin{remark} \label{rmk:qFp}
  Observe that our generalized hypergeometric function is a terminating 
balanced series 
  because $(-L)_k=0$ if $k>L$. Note also that using $(-x)_k=(-1)^k (x-k+1)_k$ we have
\[
(-L)_k=(-1)^k (L-k+1)_k=(-1)^k \frac{\Gamma(L+1)}{\Gamma(L-k+1)},\;\; 0\le k 
\le L.
\]
We drop the dependence on the dimension $d$ and write 
\begin{align}\label{hypergeometricfunction}
F_L(s) = & {}_{4}F_{3}\left( 
-L,d+L,\frac{d-s}{2},-\frac{s}{2};\frac{d}{2}+1,d-\frac{s}{2}+L,-\frac{s}{2}
-L;1  \right) \nonumber
\\
=
&
\sum_{k=0}^L \frac{(-L)_k (d+L)_k (\frac{d-s}{2})_k (-\frac{s}{2})_k  
}{(\frac{d}{2}+1)_k (d-\frac{s}{2}+L)_k (-\frac{s}{2}-L)_k}\frac{1}{k!}.
\end{align}
Observe that, as a function of the variable $s$, $F_L(0)=F_L(d)=1$. 
\end{remark}

By induction it is not difficult to show the following

\begin{proposition}
 For $0<s\le d$, $L\ge 1$ and $0\le k\le L$ we have 
\begin{equation} 						
\label{ineq_hypergeometric}
 0< \frac{(-L)_k (d+L)_k }{(d-\frac{s}{2}+L)_k (-\frac{s}{2}-L)_k}\le 1,
\end{equation}
and the quotient is decreasing in $k$.
\end{proposition}

From this proposition it is easy to deduce that when $\left(  -\frac{s}{2} 
\right)_k<0$ for all $k\ge 1$, for example
when $0<s<2$, we have
\[
{}_{2}F_{1}\left( \frac{d-s}{2},-\frac{s}{2};\frac{d}{2}+1;1  \right)\le 
F_L(s)\le 1,
\] 
  where ${}_{2}F_{1}$, is the (standard) hypergeometric function. By Gauss 
theorem (as $s>-1$)
\[
{}_{2}F_{1}\left( \frac{d-s}{2},-\frac{s}{2};\frac{d}{2}+1;1  
\right)=\frac{\Gamma\left( 1+\frac{d}{2} \right)\Gamma\left( 1+ s \right)}
{\Gamma\left( 1+\frac{s}{2} \right)\Gamma\left( 1+\frac{d+s}{2} \right)}.
\]

In fact, when $s$ is even, $s=2m$ for some $m\in \N$, the sum in the 
generalized hypergeometric function \eqref{hypergeometricfunction}
is just up to $m$ because 
$\left(  -\frac{s}{2} \right)_k=\left(  -m \right)_k=0$, when $k>m$.

From the asymptotic property of the gamma function 
\[
\lim_{n\to \infty} \frac{\Gamma(n+\alpha)}{\Gamma(n)n^\alpha}=1,\;\;\alpha\in 
\R,
\]
and the observation above, it follows that for even $s$
\begin{equation}  \label{limithypergeometric}
 F_L(s)\to {}_{2}F_{1}\left( \frac{d-s}{2},-\frac{s}{2};\frac{d}{2}+1;1  
\right)
\end{equation}
when $L$ goes to $+\infty$. We will see from our next results that this 
limit holds for all $0<s<d$. 

%===============================================================================

Now we prove the asymptotic expansion of Riesz $s$-energy. The main ingredient 
in the proof of Theorem~\ref{th:asymptoticrieszsharmonic} is the following 
estimate in terms of Bessel functions $J_\nu$ of the first kind. Recall that 
for $\nu\in\C$, Bessel functions of the first kind of order $\nu$ are the canonical solutions of the 
second order differential equation
\[
 \frac{d^2 y}{dz^2}+\frac{1}{z}\frac{d y}{dz}+ 
\left(1-\frac{\nu^2}{z^2}\right)y=0.
\]
% There are many well--known formulas for $J_\nu$, see for example 
% \cite[Sec.~8.40]{losrusos}.
\begin{proposition}\label{pr:integralestimate}
\begin{align*}
&\lim_{L\to \infty} \frac{1}{L^s}\int_{-1}^1 P_L^{(1+\lambda,\lambda)}(t)^2 
(1-t)^{\lambda-\frac{s}{2}}(1+t)^{\lambda}\,dt=
2^{\frac{s}{2}+d}\int_0^{\infty} \frac{J_{1+\lambda}(x)^2}{x^{1+s}} dx
\\
&=2^{\frac{s}{2}-1}\frac{\Gamma\left( \frac{d-s}{2} \right) }
{ \Gamma\left( 1+\frac{d}{2} \right)\Gamma\left( 1+\frac{s}{2} \right)}
{}_{2}F_{1}\left( \frac{d-s}{2},\frac{d+1}{2};d+1;1  \right)
% \\
% &=2^{\frac{s}{2}-1}\frac{\Gamma\left( \frac{d-s}{2} \right)  \Gamma\left( 
% d+1 
% \right) \Gamma\left( \frac{1+s}{2} \right) }
% { \Gamma\left( 1+\frac{d}{2} \right)\Gamma\left( 1+\frac{s}{2} \right) 
% \Gamma\left( 1+\frac{s+d}{2} \right)\Gamma\left( \frac{d+1}{2} \right) }
\\
&=2^{d+\frac{s}{2}-1} \frac{\Gamma\left( \frac{d-s}{2} \right)  \Gamma\left( 
\frac{1+s}{2} \right) }
{ \sqrt{\pi} \Gamma\left( 1+\frac{s}{2} \right) \Gamma\left( 1+\frac{s+d}{2} 
\right)}.
\end{align*}
\end{proposition}

Observe that from this Proposition we get the following

\begin{corollary}
  Given $F_L(s)$ as in \eqref{hypergeometricfunction} then for $0<s<d$
\[
\lim_{L\to +\infty} F_L(s)={}_{2}F_{1}\left( 
\frac{d-s}{2},-\frac{s}{2};\frac{d}{2}+1;1 \right).
\]
\end{corollary}

\begin{proof}
 From Lemma \ref{lem:integral2} we get that
\[
F_L(s)=\frac{\Gamma\left(1+\frac{d}{2}\right)\Gamma\left(1+\frac{s}{2}\right)}
{ 2^{d-s/2-1} C_{s,d}(L)\Gamma\left(\frac{d-s}{2}  \right)}
\int_{-1}^1 P_L^{(1+\lambda,\lambda)}(t)^2 
(1-t)^{\lambda-\frac{s}{2}}(1+t)^{\lambda}\,dt. 
\]
Now from the Proposition above and Legendre's duplication formula
\begin{equation} \label{eq:duplication}
 \sqrt{\pi} \Gamma(x)=2^{x-1}\Gammaf{x+1}{2}\Gammaf{x}{2}
\end{equation}
we get the result.
\end{proof}

\begin{proof}[Proof of Proposition \ref{pr:integralestimate}]
  The second equality is from \cite[p.47, (4)]{EOT54}. For the last equality 
use Gauss theorem about the hypergeometric function
and the duplication formula \eqref{eq:duplication}.

For the first equality, we split the integral
\begin{align*}
 \int_{-1}^1 & L^{-s}  P_L^{(1+\lambda,\lambda)}(t)^2 
(1-t)^{\lambda-\frac{s}{2}}(1+t)^{\lambda}\,dt
\\
=&\left[\int_{-1}^{-\cos \frac{c}{L}}+\int_{-\cos \frac{c}{L}}^{\cos 
\frac{c}{L}}+
\int_{\cos \frac{c}{L}}^1 \right] L^{-s} P_L^{(1+\lambda,\lambda)}(t)^2 
(1-t)^{\lambda-\frac{s}{2}}(1+t)^{\lambda}\,dt
\\
=& A(c,L)+B(c,L)+C(c,L),
\end{align*}
where $c>0$ is fixed and $c<\pi L$. For the boundary parts we do a change of 
variables $t=\cos(x/L)$ to get
\[
\begin{split}
C(c,L)&=\\
&2^{s/2}\int_0^c L^{-2-2\lambda}  P_L^{(1+\lambda,\lambda)}\left(\cos 
\frac{x}{L}\right)^2 \left( \frac{\sin \frac{x}{L}}{\frac{x}{L}} 
\right)^{2\lambda+1}
\left( \frac{1-\cos \frac{x}{L}}{\frac{1}{2}\left(\frac{x}{L}\right)^2} 
\right)^{-s/2} x^{2\lambda+1-s}\,dx.
\end{split}
\]
Using the Mehler-Heine estimates \eqref{mehlerheine} and the elementary limits
\[
 \lim_{L\rightarrow\infty}\frac{\sin \frac{x}{L}}{\frac{x}{L}}=1,\quad 
\lim_{L\rightarrow\infty}\frac{1-\cos 
\frac{x}{L}}{\frac{1}{2}\left(\frac{x}{L}\right)^2}=1,
\]
we conclude:
\[
 \lim_{L\rightarrow\infty}C(c,L)=2^{\frac{s}{2}+d}\int_0^c 
\frac{J_{1+\lambda}(x)^2}{x^{1+s}} dx.
\]

For the other end of the interval, using the change of variables $t=-\cos(x/L)$ 
we get
\[
A(c,L)=
\int_0^c L^{-2-2\lambda}  P_L^{(1+\lambda,\lambda)}\left(-\cos 
\frac{x}{L}\right)^2 
\left( \frac{\sin \frac{x}{L}}{\frac{x}{L}} \right)^{2\lambda+1}
\left(\frac{ 1+\cos \frac{x}{L} }{\left(\frac{x}{L}\right)^2}\right)^{-s/2} 
x^{2\lambda+1} dx,
\]
and, again using \eqref{mehlerheine}, this expression converges to zero when $L\to \infty$.
For the middle integral we get, after the change of variables $t=-\cos \theta$ 
and the use of the asymptotic estimate for Jacobi polynomials 
\eqref{Jacobiestimate}
\begin{align*}
0\leq B(c,L) & \lesssim \frac{1}{L^{s+1}}\int_{\frac{c}{L}}^{\pi-\frac{c}{L}} 
\frac{1}{(\sin \frac{\theta}{2})^{s+2}}d\theta\lesssim
\frac{1}{L^{s+1}}\int_{\frac{c}{L}}^{\pi-\frac{c}{L}} 
\frac{1}{\theta^{s+2}}d\theta
\\
&
\le 
\frac{1}{L^{s+1}}\int_{\frac{c}{L}}^{+\infty} 
\frac{1}{\theta^{s+2}}d\theta=\frac{1}{(s+1)c^{s+1}}.
\end{align*}
We have then proved that for all $c>0$,
\begin{align*}
 2^{\frac{s}{2}+d}\int_0^c \frac{J_{1+\lambda}(x)^2}{x^{1+s}} 
dx\leq&\lim_{L\rightarrow\infty} \left( \int_{-1}^1  L^{-s}  
P_L^{(1+\lambda,\lambda)}(t)^2 
(1-t)^{\lambda-\frac{s}{2}}(1+t)^{\lambda}\,dt\right)\\ \leq & 
\frac{R(s,d)}{(s+1)c^{s+1}}+ 2^{\frac{s}{2}+d}\int_0^\infty 
\frac{J_{1+\lambda}(x)^2}{x^{1+s}}\, dx,
\end{align*}
with $R(s,d)$ a constant independent of $c$. Taking the limit as 
$c\rightarrow\infty$, the result follows.
\end{proof}

%===============================================================================

\begin{proof}[Proof of Theorem~\ref{th:asymptoticrieszsharmonic}]

  From Theorem \ref{th:rieszsharmonic} we get that
\[
\mathbb{E}_{x\in(\S^d)^n}(E_s(x))=\pi_L^2 V_s(\S^d)-
\frac{\pi_L^2 \omega_{d-1}}{\binom{L+\frac{d}{2}}{L}^2 
2^{s/2}\omega_d}\int_{-1}^1 P_L^{(1+\lambda,\lambda)}(t)^2 
(1-t)^{\lambda-\frac{s}{2}}(1+t)^{\lambda}\,dt.
\]
  When $L\to \infty$
\[
\frac{\pi_L}{\binom{L+\frac{d}{2}}{L}^2}= \frac{2 \Gamma\left( 1+\frac{d}{2} 
\right)^2 }{d!}+o(1)\;\;\mbox{and}\;\;\pi_L^{s/d}=\left( 
\frac{2}{d!}\right)^{s/d} L^s (1+o(1)).
\] 
From 
\begin{align*}
\mathbb{E}_{x\in(\S^d)^n}(E_s(x)) & -\pi_L^2 V_s(\S^d)=
\pi_L^{1+\frac{s}{d}}\left(\frac{d!}{2}\right)^{-1+\frac{s}{d}}
\Gamma\left( 1+\frac{d}{2} \right)^2  \frac{\omega_{d-1}}{2^{s/2}\omega_d}
\\
&
\times (1+o(1))
\int_{-1}^1 L^{-s} P_L^{(1+\lambda,\lambda)}(t)^2 
(1-t)^{\lambda-\frac{s}{2}}(1+t)^{\lambda}\,dt,
\end{align*}
and Proposition \ref{pr:integralestimate} we get the result with
\[ 
C_{s,d}=\left(\frac{d!}{2}\right)^{-1+\frac{s}{d}} 
\frac{\omega_{d-1}}{\omega_d}2^{d-1} \frac{ \Gamma\left( 1+\frac{d}{2} \right)^2 
\Gamma\left( \frac{d-s}{2} \right)  \Gamma\left( \frac{1+s}{2} \right) } { 
\sqrt{\pi} \Gamma\left( 1+\frac{s}{2} \right) \Gamma\left( 1+\frac{s+d}{2} 
\right)}.
\]
% 
% \begin{align*}
% C_{s,d}=&\left(\frac{d!}{2}\right)^{-1+\frac{s}{d}}
%   \frac{\omega_{d-1}}{\omega_d}2^{d-1} 
% \frac{ \Gamma\left( 1+\frac{d}{2} \right)^2 \Gamma\left( \frac{d-s}{2} \right)  
% \Gamma\left( \frac{1+s}{2} \right) }
% { \sqrt{\pi} \Gamma\left( 1+\frac{s}{2} \right) \Gamma\left( 1+\frac{s+d}{2} 
% \right)}
% \\
% =& \left(\frac{d!}{2}\right)^{-1+\frac{s}{d}}
% \frac{ d\,\Gamma\left( 1+\frac{d}{2} \right)   \Gamma\left( \frac{1+s}{2} 
% \right) }
% { 2\sqrt{\pi} \Gamma\left( 1+\frac{s}{2} \right) \Gamma\left( 1+\frac{s+d}{2} 
% \right)}\frac{\omega_{d-1}}{\omega_d}2^{d-1} \Gamma\left( \frac{d}{2} 
% \right)\Gamma\left( \frac{d-s}{2} \right).
% \end{align*}
Finally, recall the value of $V_s(\S^d)$ from \eqref{eq:Vs} which yields
\[
 \frac{\omega_{d-1}}{\omega_d}2^{d-1} \Gamma\left( \frac{d}{2} 
\right)\Gamma\left( \frac{d-s}{2} \right) = 
2^sV_s(\S^d)\Gamma\left(d-\frac{s}{2}\right).
\]
\end{proof}

%===============================================================================
%===============================================================================
%===============================================================================

%\subsection{Logarithmic energy}

\begin{proof}[Proof of Theorem~\ref{th:logharmonic}]

We compute the derivative at $s=0$ of the expression given in Theorem 
\ref{th:rieszsharmonic}, which according to Corollary \ref{cor:rieszenergies} 
equals the expected value of the logarithmic energy. We then have (using 
\eqref{eq:VlogVs} for the first term):
\[
\mathbb{E}_{x\in (\S^d)^n} (E_0(x))=
% \pi_L^2 
% V_{\mathrm log}(\S^d)-Q\frac{d}{ds}\mid_{s=0}\left(A(s)B(s)C(s)\right)=
\pi_L^2 
V_{\log}(\S^d)-C_L \frac{d}{ds}\mid_{s=0^+}(D_L(s)E_L(s)F_L(s)),
\]
where 
\begin{align*}
 C_L\;=&\;\frac{2^{d-1}\omega_{d-1}\pi_L^2\Gamma\left(  L+\frac{d}{2} 
\right)\Gamma\left(  L+\frac{d}{2}+1 \right) 
}{\binom{L+\frac{d}{2}}{L}^2\omega_d \Gamma(L+1)^2 \Gamma\left( 1+\frac{d}{2} 
\right)},\\
 D_L(s)\;=&\;\frac{\Gamma\left( L+\frac{s}{2}+1 \right)}{\Gamma\left( 
1+\frac{s}{2} \right)},\\
 E_L(s)\;=&\; \frac{ 2^{-s}\Gamma\left( \frac{d-s}{2} \right) }{\Gamma\left( 
L-\frac{s}{2}+d \right)}, \\
 F_L(s)\;=&\;{}_{4}F_{3}\left( 
-L,d+L,\frac{d-s}{2},-\frac{s}{2};\frac{d}{2}+1,d-\frac{s}{2}+L,-\frac{s}{2}
-L;1  \right)\\
 \;=&\;\sum_{k=0}^L \frac{(-L)_k (d+L)_k (\frac{d-s}{2})_k (-\frac{s}{2})_k  
}{(\frac{d}{2}+1)_k (d-\frac{s}{2}+L)_k (-\frac{s}{2}-L)_k}\frac{1}{k!}.
\end{align*}

We have that
\[
D_L(0)=L!,\;\;E_L(0)=\frac{\Gamma\left(  \frac{d}{2} \right)}{\Gamma\left(  d+L 
\right)},\;\;F_L(0)=1.
\]
The derivatives at $0$ of $D_L$ and $E_L$ are:
\begin{align*}
D_L'(0)&=\frac{1}{2}\left( \Gamma'\left(  L+1 \right)-\Gamma\left(  L+1 
\right)\Gamma'\left(  1 \right)  \right)=\frac{L!}{2}H_L,
\\
E_L'(0)&=\frac{B(0)}{2}\left( H_{L+d-1}-\gamma-\psi_0(d/2)-2\log 2  \right),
\end{align*}
where $\gamma=-\psi_0(1)=-\Gamma'(1)$ is the Euler-Mascheroni constant.

Finally, from the series expression of $F_L(s)$ and using that $(0)_k=0$ when 
$k>0$ we deduce that 
\[F_L'(0)=\sum_{k=1}^L \frac{(-L)_k (d+L)_k (\frac{d}{2})_k 
\frac{d}{ds}[(-\frac{s}{2})_k]_{s=0}  }{(\frac{d}{2}+1)_k (d+L)_k 
(-L)_k}\frac{1}{k!}=
\sum_{k=1}^L 
\frac{d}{ds}\Bigl[\bigl(-\frac{s}{2}\bigr)_k\Bigr]_{s=0}\frac{d}{d+2k}\frac{1}{
k! }. 
\]
For $k\geq$ we have:
\[
\frac{d}{ds}\Bigl[\bigl(-\frac{s}{2}\bigr)_k\Bigr]=-\frac{1}{2}\left(-\frac{s}{2
}
\right)_k\left( 
\psi_0\left(k-\frac{s}{2}\right)-\psi_0\left(-\frac{s}{2}\right)\right)
\to -\frac{1}{2}\Gamma(k),\;s\to 0^+,
\]
and conclude that
\[
F_L'(0)=-\frac{d}{2}\sum_{k=1}^L \frac{1}{k(d+2k)}.
\]
% Observe that when $d=2$ we get $C'(0)=-1/2\left(1-\frac{1}{L+1}\right)$.

We have then proved that
\begin{align*}
& \frac{d}{ds}\mid_{s=0^+} \left(D_L(s)E_L(s)F_L(s)\right)
\\
 &
 =\frac{\Gamma\left(  \frac{d}{2} \right)L!}{2\Gamma\left(  d+L 
\right)}\left(H_L+H_{L+d-1}+\psi_0(1/2)-\psi_0(d/2)-d\sum_{k=1}^L 
\frac{1}{k(d+2k)}\right).\\
\end{align*}
Note that
\[d\sum_{k=1}^L \frac{1}{k(d+2k)}=\sum_{k=1}^L \left( \frac{1}{k} 
-\frac{1}{\frac{d}{2}+k}\right)=H_L-\sum_{k=1}^L \frac{1}{\frac{d}{2}+k}.\]
We have proved that
\begin{align*}
  & \mathbb{E}_{x\in (\S^d)^n}  (E_0(x))
 \\
 =
  &
  \pi_L^2V_{\log}(\S^d)-\frac{C_L \Gamma\left(  
\frac{d}{2} \right)L!}{2\Gamma\left(  d+L \right)}
\left( \sum_{k=1}^L \frac{1}{\frac{d}{2}+k}+H_{L+d-1}+\psi_0(1/2)-\psi_0(d/2)   
\right).
\end{align*}
Finally, by using the duplication formula (\ref{eq:duplication}) one can easily check that
\[
 \frac{C_L \Gamma\left(  \frac{d}{2} \right)L!}{\Gamma\left(  d+L \right)}=\pi_L.
\]
% Indeed,
% \begin{align*}
%  \frac{C_L\Gamma\left(  \frac{d}{2} \right)L!}{\Gamma\left(  d+L 
% \right)}\;=&\;\frac{2^{d-1}\omega_{d-1}\pi_L^2\Gamma\left(  L+\frac{d}{2} 
% \right)\Gamma\left(  L+\frac{d}{2}+1 \right)\Gamma\left(  \frac{d}{2} \right)L! 
% }{\binom{L+\frac{d}{2}}{L}^2\omega_d \Gamma(L+1)^2 \Gamma\left( 1+\frac{d}{2} 
% \right)\Gamma\left(  d+L \right)}\\
%  \;\stackrel{\eqref{eq:omegad}}{=}&\;\frac{2^{d-1}\Gammaf{d+1}{2}\Gamma\left(  
% L+\frac{d}{2} \right)\Gamma\left(  L+\frac{d}{2}+1 \right)L! 
% }{\sqrt{\pi}\binom{L+\frac{d}{2}}{L}^2\Gamma(L+1)^2 \Gamma\left( 1+\frac{d}{2} 
% \right)\Gamma\left(  d+L \right)}\pi_L^2\\
%  \;{=}&\;\frac{2^{d-1}\Gammaf{d+1}{2}\Gamma\left(  L+\frac{d}{2} 
% \right)\Gamma\left(  L+\frac{d}{2}+1 \right)L!^3\Gamma\left( 1+\frac{d}{2} 
% \right)}{\sqrt{\pi}\Gamma\left(L+\frac{d}{2}+1\right)^2\Gamma(L+1)^2 
% \Gamma\left(  d+L \right)}\pi_L^2\\
%  \;{=}&\;\frac{2^{d-1}\Gammaf{d+1}{2}\Gamma\left( 1+\frac{d}{2} 
% \right)L!}{\sqrt{\pi}\left(L+\frac{d}{2}\right) \Gamma\left(  d+L 
% \right)}\pi_L^2\\ 
% \;{=}&\;\frac{2^{d-1}\Gammaf{d+1}{2}\Gammaf{d}{2}}{\sqrt{\pi}}\frac{dL!}{
% \left(2L+d\right) \Gamma\left(  d+L \right)}\pi_L^2\\
%  \;\stackrel{\eqref{eq:duplication}}{=}&\;\frac{d(d-1)!L!}{\left(2L+d\right) 
% \left(  d+L -1\right)!}\pi_L^2\;\stackrel{\eqref{eq:piL}}{=}\;\pi_L,
%  \end{align*}
%  as wanted. 
%Theorem~\ref{th:logharmonic} is now proved.
\end{proof}

%===============================================================================

\begin{proof}[Proof of Corollary~\ref{cor:logharmonic}]
We compute the asymptotic behavior based on the equality of 
Theorem~\ref{th:logharmonic}. 
% Observe that
% \[\pi_L=\frac{2L+d}{d}\frac{\Gamma\left( d+L \right)}{\Gamma\left( L+1 
%\right)\Gamma\left( d \right)}=\frac{2L+d}{d \Gamma\left( d \right)
% }\left( L^{d-1}+\frac{d(d-1)}{2}L^{d-2}+\dots\right)\[
% \[=
% \frac{2}{d \Gamma\left( d \right)}L^d+\frac{d}{\Gamma\left( d 
%\right)}L^{d-1}+O(L^{d-2})\]
  From the recurrence relation
\[
\psi_0(x+1)=\psi_0(x)+\frac{1}{x}
\]
  we get that
\[
\sum_{k=1}^L \frac{1}{\frac{d}{2}+k}=
% \psi_0\left(  1+L+\frac{d}{2} 
% \right)-\psi_0\left(  1+\frac{d}{2} \right)=
\psi_0\left(  \frac{d}{2}+1+L \right)-\psi_0\left(  \frac{d}{2} 
\right)-\frac{2}{d}
\]
  Recall the asymptotic expansions (for $x,n\rightarrow\infty$):
\[
\psi_0(x)=\log x-\frac{1}{2x}+o(x^{-1}),\qquad H_n=\psi_0(n+1)+\gamma=
\log n +\gamma+\frac{1}{2n}+o(n^{-1})
\] 
  So when $L\rightarrow\infty$ we have from Theorem~\ref{th:logharmonic} 
($\psi_0(1/2)=-\gamma-2\log 2$)
$$
   \mathbb{E}_{x\in (\S^d)^n} (E_0(x))=\pi_L^2 V_{\log}(\S^d)- \pi_L\left(\log L-\log 2-\psi_0\left( 
\frac{d}{2} \right)-\frac{1}{d}+o(1)\right).
$$
  From the asymptotic expression \eqref{eq:piL} we have as $L\rightarrow\infty$:
  \begin{equation}\label{eq:Lasymp}
   \log L=\frac{\log \pi_L}{d}-\frac{\log(2/d!)}{d}+o(1),
  \end{equation}
  so we have proved:
\begin{align*}
 \mathbb{E}_{x\in (\S^d)^n} (E_0(x))\;=&\; 
\pi_L^2 V_{\log}(\S^d)-\frac{1}{d}\pi_L\log \pi_L\\
&\qquad \;+\;\left( \frac{1}{d}\log \frac{2}{d!}+\log 2+\psi_0\left( 
\frac{d}{2} \right)+\frac{1}{d}\right) \pi_L+o(\pi_L),
\end{align*}
as wanted.
\end{proof}

To prove Theorem~\ref{th:rieszsingular} we use the following result.

\begin{proposition}			\label{limitderivative}
  We have that
\begin{align*}
 \lim_{L\to \infty} & \frac{d}{ds} \left[{}_{4}F_{3}\left( 
-L,d+L,\frac{d-s}{2},-\frac{s}{2};\frac{d}{2}+1,d-\frac{s}{2}+L,-\frac{s}{2}
-L;1  \right)\right]_{s=d}
 \\
 &
 =
 \frac{d}{ds} \left[ {}_{2}F_{1}\left( 
\frac{d-s}{2},-\frac{s}{2};\frac{d}{2}+1;1  \right) \right]_{s=d}=
 -\frac{1}{2}\sum_{k=1}^{\infty}\frac{ (-\frac{d}{2})_k  
}{(\frac{d}{2}+1)_k}\frac{1}{k}
 \\
 &
 =\frac{1}{2} \left(\psi_0(d+1)-\psi_0\left(\frac{d}{2}+1\right)\right).
\end{align*}
\end{proposition}

%===============================================================================

\begin{proof}
Let as before
\[
F_L(s)={}_{4}F_{3}\left( 
-L,d+L,\frac{d-s}{2},-\frac{s}{2};\frac{d}{2}+1,d-\frac{s}{2}+L,-\frac{s}{2}
-L;1  \right).
\]
Then from the fact that
\[
\lim_{s\to d^-}\frac{d}{ds}\left[\left(\frac{d-s}{2}\right)_k\right]= 
-\frac{\Gamma(k)}{2},
\]
we get that
\[F_L'(d)=-\frac{1}{2}\sum_{k=1}^L 
\frac{(-L)_k (d+L)_k (-\frac{d}{2})_k  }{(\frac{d}{2}+1)_k (\frac{d}{2}+L)_k 
(-\frac{d}{2}-L)_k}\frac{1}{k}.
\]

Observe that, as in the discussion before Proposition \ref{pr:integralestimate},
when $d$ is even, the sum in the generalized hypergeometric function is up to $d/2$ and 
then as for all $k=1,\dots , d/2$,
\[
\lim_{L\to +\infty} \frac{(-L)_k (d+L)_k }{(\frac{d}{2}+L)_k 
(-\frac{d}{2}-L)_k}=1,
\]
we get that
\[
\lim_{L\to +\infty} F_L'(d)=-\frac{1}{2}\sum_{k=1}^{d/2}\frac{ 
(-\frac{d}{2})_k  }{(\frac{d}{2}+1)_k}\frac{1}{k}=
\frac{d}{ds} \left[{}_{2}F_{1}\left( \frac{d-s}{2},-\frac{s}{2};\frac{d}{2}+1;1 
\right)\right]_{s=d}
\]

For odd $d$ also $F_L'(d)$ converges when $L\to +\infty$ to 
\[
\frac{d}{ds} \left[ {}_{2}F_{1}\left( 
\frac{d-s}{2},-\frac{s}{2};\frac{d}{2}+1;1  \right) 
\right]_{s=d}=-\frac{1}{2}\sum_{k=1}^{\infty}\frac{ (-\frac{d}{2})_k  
}{(\frac{d}{2}+1)_k}\frac{1}{k}.
\]
Indeed, we have that for $1\le k\le L$,
\[
\frac{ (-\frac{d}{2})_k  }{(\frac{d}{2}+1)_k}=(-1)^k\frac{\Gamma\left( 
\frac{d}{2}+1 \right)^2}{\Gamma\left( \frac{d}{2}+k+1 \right)\Gamma\left( 
\frac{d}{2}-k+1 \right)}.
\]
By Euler's reflection formula for the Gamma function and Stirling approximation 
we get that
\[
\left|\Gamma\left( \frac{d}{2}+k+1 \right)\Gamma\left( \frac{d}{2}-k+1 
\right)\right|\sim \pi k^d,
\]
and therefore
\[
\sum_{k=1}^{\infty}\left| \frac{ (-\frac{d}{2})_k  
}{(\frac{d}{2}+1)_k}\frac{1}{k}\right|\lesssim \sum_{k=1}^{\infty} 
\frac{1}{k^{d+1}}.
\]
Given $\epsilon>0,$ we choose $n_0$ such that
\[
\sum_{k=n_0+1}^{\infty}\left| \frac{ (-\frac{d}{2})_k  
}{(\frac{d}{2}+1)_k}\frac{1}{k}\right|<\epsilon,
\]
and we get from inequality \eqref{ineq_hypergeometric} that for $L>>n_0$
\[
\left| \sum_{k=n_0+1}^L 
\frac{(-L)_k (d+L)_k (-\frac{d}{2})_k  }{(\frac{d}{2}+1)_k (\frac{d}{2}+L)_k 
(-\frac{d}{2}-L)_k}\frac{1}{k}\right|<\epsilon.
\]

Now the result follows as before from the fact that for all $k=1,\dots , n_0$,
\[
\lim_{L\to +\infty} \frac{(-L)_k (d+L)_k }{(\frac{d}{2}+L)_k 
(-\frac{d}{2}-L)_k}=1,
\]
and therefore 
\[-\frac{1}{2}\sum_{k=1}^{n_0} 
\frac{(-L)_k (d+L)_k (-\frac{d}{2})_k  }{(\frac{d}{2}+1)_k (\frac{d}{2}+L)_k 
(-\frac{d}{2}-L)_k}\frac{1}{k}\to
\sum_{k=1}^{n_0} \frac{ (-\frac{d}{2})_k  }{(\frac{d}{2}+1)_k}\frac{1}{k},\]
when $L\to +\infty$.

The last equality follows from Gauss theorem

\begin{align*}
\frac{d}{ds} & \left[{}_{2}F_{1}\left( 
\frac{d-s}{2},-\frac{s}{2};\frac{d}{2}+1;1  \right)\right]_{s=d}
=\frac{1}{2}\frac{\Gamma'\left( d+1 \right)\Gamma\left( \frac{d}{2}+1 
\right)-\Gamma\left( d+1 \right)\Gamma'\left( \frac{d}{2}+1 
\right)}{\Gamma\left( d+1 \right)\Gamma\left( \frac{d}{2}+1 \right)}
\\
&
=\frac{1}{2}(\psi_0(d+1)-\psi_0(\frac{d}{2}+1)). 
\end{align*}
\end{proof}

%===============================================================================

\begin{proof}[Proof of Theorem~\ref{th:rieszsingular}]

The function $V_s(\S^d)$ is meromorphic with a simple pole in $s=d$
(because of the term $\Gamma(\frac{d-s}{2})$) and the residue in $d$ equals 
\[
\lim_{s\to d} V_s(\S^d)(s-d)=-\frac{\omega_{d-1}}{\omega_{d}}.
\]

  We can write from Theorem~\ref{th:rieszsharmonic} 
 \[
  \mathbb{E}_{x\in(\S^d)^n}(E_s(x))=n^2 V_s(\S^d)\left( 1-U(s)    \right),
 \]
  where
  \[U(s)=\frac{d F_L(s) 
  \Gamma\left( d-\frac{s}{2} \right)\Gamma\left( L+\frac{s}{2}+1 
\right)}{(2L+d)\Gamma\left( 1+\frac{s}{2} \right)\Gamma\left( L-\frac{s}{2}+d 
\right)},\]
  and
 $F_L(s)={}_{4}F_{3}\left( 
-L,d+L,\frac{d-s}{2},-\frac{s}{2};\frac{d}{2}+1,d-\frac{s}{2}+L,-\frac{s}{2}
-L;1  \right)$.
   As $U(d)=1$ we have that 
\[
\lim_{s\to d} \mathbb{E}_{x\in(\S^d)^n}(E_s(x))= 
\frac{\omega_{d-1}}{\omega_d}\pi_L^2 U'(d).
\]
   We compute this derivative by writing $U(s)=\frac{d}{2L+d}F_L(s)G_L(s)$
  with
\[G_L(s)=\frac{(\frac{s}{2}+1)_L}{(d-\frac{s}{2})_L},\]
  and using  
\[G_L'(d)=\frac{2L+d}{d}\left( 
\psi_0\left(\frac{d}{2}+L\right)-\psi_0\left(\frac{d}{2}\right)+\frac{1}{d+2L}
-\frac{1}{d} \right),\]
  and the asymptotic formula for the digamma function
\[
\psi_0\left(\frac{d}{2}+L\right)+\frac{1}{d+2L}
=\frac{1}{d}\log \pi_L-\frac{1}{d}\log\frac{2}{d!}+o(1).
\]
 Finally, we get
\[
\lim_{s\to d} \mathbb{E}_{x\in(\S^d)^n}(E_s(x))=\frac{\omega_{d-1}}{d 
\omega_d}\pi_L^2 \log \pi_L
\]
\[
+\pi_L^2 \left( 
\frac{\omega_{d-1}}{\omega_d}F'(d)-\psi_0\left(\frac{d}{2}\right)-\frac{1}{d}
-\frac{1}{d}\log\frac{2}{d!}+o(1)\right),
\]
where
\[F_L'(d)=\frac{d}{ds} \left[{}_{4}F_{3}\left( 
-L,d+L,\frac{d-s}{2},-\frac{s}{2};\frac{d}{2}+1,d-\frac{s}{2}+L,-\frac{s}{2}
-L;1  \right)\right]_{s=d}.\]
The result then follows from Proposition \ref{limitderivative}.
\end{proof}

%===============================================================================
%===============================================================================
%===============================================================================

%\section{General kernels in the sphere which depend only on the distance}

\subsection{Optimality among isotropic kernels. Theorems~\ref{th:rieszsgeneral1}, \ref{th:case2} and \ref{th:optimal}}

\begin{proof}[Proof of Theorem~\ref{th:rieszsgeneral1}]
 
From Corollary~\ref{cor:rieszenergies} we have
\begin{align*}
 \mathbb{E}_{x\in(\Sd)^n}(E_s(x))=%&\int_{u,v\in\Sd}\frac{n^2-|K(\langle 
%u,v\rangle)|^2}{\|u-v\|^s}\,d\mu(u)\,d\mu(v)\\
%=&
 \int_{u,v\in\Sd}\frac{n^2-|K(\langle u,v\rangle)|^2}{(2-2\langle 
u,v\rangle)^{s/2}}\,d\mu(u)\,d\mu(v).
\end{align*}
This last integral is invariant if we rotate simultaneously $u$ and $v$, thus 
we can assume that $v=e_{1}=(1,0,\ldots,0)$ and we get
\begin{align*}
 \mathbb{E}_{x\in(\Sd)^n}(E_s(x))=&
\int_{u\in\Sd}\frac{n^2-|K(u)|^2}{(2-2u)^{s/2}}\,d\mu(u)\\= 
&\frac{\omega_{d-1}}{ \omega_d}\int_{-1}^1 
\frac{n^2-|K(t)|^2}{(2-2t)^{s/2}}(1-t^2)^{d/2-1}\,dt,
\end{align*}
the last equality from Lemma~\ref{lem:computeintegralsphere}. The theorem 
follows noting that
 \[
  \int_{-1}^1\frac{(1-t^2)^{d/2-1}}{(1-t)^{s/2}}\,dt=2^{d-1-\frac{s}{2}} 
{\mathrm B}\left(\frac{d}{2},\frac{d}{2}-\frac{s}{2}\right).
 \]
\end{proof}

%===============================================================================

\begin{remark}
We readily have:
$$
 \mathbb{E}_{x\in(\Sd)^n}(E_s(x))\leq \frac{\omega_{d-1}}{\omega_d}n^2 2^{d-s-1} {\mathrm B}\left(\frac{d}{2},\frac{d}{2}-\frac{s}{2}\right)=n^2V_s(\S^d).
$$
% \begin{align*}
%  \mathbb{E}_{x\in(\Sd)^n}(E_s(x))\leq &\frac{\omega_{d-1}}{\omega_d}n^2 
% 2^{d-s-1} {\mathrm 
% B}\left(\frac{d}{2},\frac{d}{2}-\frac{s}{2}\right)\\&=
% n^2\,2^{d-s-1}\frac{
% \Gamma\left(\frac{d+1}{2}\right)\Gamma\left(\frac{d-s}{2}\right)}{\sqrt{\pi}
% \Gamma\left(d-\frac{s}{2}\right)}\\&=n^2V_s(\S^d).
% \end{align*}
\end{remark}
\begin{remark}
We can substitute $K(t)$ by its expansion \eqref{eq:expansionesferas} to get a 
formula for the second integral in Theorem~\ref{th:rieszsgeneral1}:
\[
 \int_{-1}^1\frac{|K(t)|^2(1-t^2)^{d/2-1}}{(1-t)^{s/2}}\,dt=
 \]
 \begin{equation}\label{eq:here}
 \sum_{j,k=0}^\infty a_ka_j\int_{-1}^1(1-t)^{d/2-1-s/2}(1+t)^{d/2-1} 
C_k^{\frac{d-1}{2}}(t)C_j^{\frac{d-1}{2}}(t)\,dt
\end{equation}
\[
 =\sum_{j,k=0}^\infty 
a_ka_j\frac{2^{d-1-s/2}\Gammaf{d-s}{2}\Gammaf{d}{2}\Gammaf{2j+s}{2}
\Gamma(d-1+j)\Gamma(d-1+k)}{\Gamma(j+1) 
\Gamma(k+1)\Gammaf{s}{2}\Gammaf{2d+2j-s}{ 2}\Gamma(d-1)^2}\times
\]
\[
\qquad_{4}F_{3}\left(-k,k+d-1,\frac{d-s}{2}, 
1-\frac{s}{2};\frac{d}{2},d-\frac{s} {2}+j,1-\frac{s}{2}-j;1\right).
\]
\end{remark}

%===============================================================================
%===============================================================================
%===============================================================================

\subsubsection{The special case $s=2$ and $d\ge 3.$ Theorem~\ref{th:case2}}

One can use Theorem~\ref{th:rieszsgeneral1} and 
\eqref{eq:here} to write
\begin{equation}\label{eq:esperanzaQs}
 \mathbb{E}_{x\in(\Sd)^n}(E_2(x))= \frac{\omega_{d-1}}{2\omega_d}\left(n^2 
2^{d-2} {\mathrm B}\left(\frac{d}{2},\frac{d}{2}-1\right)-\sum_{j,k=0}^\infty 
a_ka_j Q_{k,j}^{d}\right),
\end{equation}
where
\begin{align*}
  Q_{k,j}^d 
  =&\int_{-1}^1(1-t)^{d/2-2}(1+t)^{d/2-1}C_k^{\frac{d-1}{2}}(t)C_j^{\frac{d-1}{2}}(t)\,dt.
\end{align*}
The following lemma shows that these integrals can be solved exactly.
\begin{lemma}\label{lem:Q}
 Let $d\geq3$ and $s=2$. Then for all $0\leq k\leq j$ we have
 \[
  Q_{j,k}^d=Q_{k,j}^d=Q_{k,k}^d=2^{d-2}\binom{d+k-2}{k}\,\mathrm 
B\left(\frac{d}{2},\frac{d}{2}-1\right)=\binom{d+k-2}{k}Q_{0,0}^d,
 \]
\end{lemma}
\begin{remark}
 The value of the integral in Lemma~\ref{lem:Q} seems to be known just in the 
case $k=j$ (see for example \cite[p. 803]{losrusos} which gives an alternative 
but equivalent expression). Note also that we have
 \begin{equation}\label{eq:Q00omegas}
  \frac{\omega_{d-1}}{2\omega_d}Q_{0,0}^d=V_2(\S^d)
% =2^{d-3}
% \frac{\Gamma\left( \frac{d+1}{2} \right)\Gamma\left( \frac{d}{2}-1 
% \right)}{\sqrt{\pi}\Gamma\left( d-1 \right)}
=\frac{d-1}{2d-4}.
 \end{equation}
%(For the last equality, use \eqref{eq:duplication}).
\end{remark}
The proof of Lemma~\ref{lem:Q} will be a long computation. We will use the 
following basic integral, valid for $a,b>0$:
\begin{equation}\label{eq:beta} 
\int_{-1}^1(1-t)^a(1+t)^b\,dt\stackrel{t=2u-1}{=}2^{a+b+1}\int_0^1(1-u)^au^b\,
du=2^{a+b+1}\mathrm B(a+1,b+1).
\end{equation}
We will also use Legendre's duplication formula in the following form.
\begin{equation}\label{eq:legendre} 
\Gammaf{d-1}{2}=\frac{\sqrt{\pi}\Gamma(d-2)}{2^{d-3}\Gamma\left(\frac{d}{2}
-1\right)}.
\end{equation}

\begin{proof}[Proof of Lemma~\ref{lem:Q}]
We start by computing a few cases for small $k,j$. For $k=j=0$ we have:
\begin{align*}
Q_{0,0}^d=&\int_{-1}^1(1-t)^{d/2-2}(1+t)^{d/2-1}\,dt 
\stackrel{\eqref{eq:beta}}{= }2^{d-2}\mathrm 
B\left(\frac{d}{2},\frac{d}{2}-1\right).
\end{align*}
For $k=1,j=0$ we have:
\begin{align*}
 Q_{1,0}^d=&\int_{-1}^1(1-t)^{d/2-2}(1+t)^{d/2-1}(d-1)t\,dt\\
 =&-(d-1)\left(\int_{-1}^1(1-t)^{d/2-2}(1+t)^{d/2-1}(1-t-1)\,dt\right)\\
 =&-(d-1)\left(\int_{-1}^1(1-t)^{d/2-1}(1+t)^{d/2-1}\,dt-Q_{0,0}^d\right)\\
 \;\stackrel{\eqref{eq:beta}}{=}&-(d-1)\left(2^{d-1}\mathrm 
B\left(\frac{d}{2},\frac{d}{2}\right)-Q_{0,0}^d\right)=2^{d-2}\mathrm 
B\left(\frac{d}{2},\frac{d}{2}-1\right)=Q_{0,0}^d.
\end{align*}
For $k=j=1$ we have:
\begin{align*}
 Q_{1,1}^d=&\int_{-1}^1(1-t)^{d/2-2}(1+t)^{d/2-1}(d-1)^2t^2\,dt
% \\
%  =
% &-(d-1)^2\int_{-1}^1(1-t)^{d/2-2}(1+t)^{d/2-1}(1-t^2-1)\,dt
\\
 =
&-(d-1)^2\left(\int_{-1}^1(1-t)^{d/2-1}(1+t)^{d/2}\,dt-Q_{0,0}^d\right)
\\
 \;\stackrel{\eqref{eq:beta}}{=}&-(d-1)^2\left(2^d\mathrm 
B\left(\frac{d}{2},\frac{d+1}{2}\right)-Q_{0,0}^d\right)
%\\
% =&2^{d-2}(d-1)\mathrm 
% B\left(\frac{d}{2},\frac{d}{2}-1\right)
=(d-1)Q_{0,0}^d.
 \end{align*}
 We are now ready to prove the general case. Recall the recurrence relation 
satisfied by Gegenbauer's polynomials:
\[
 \ell C_\ell^{d/2-1/2}(t)=\left(2\ell+d-3\right)\,t\, 
C_{\ell-1}^{d/2-1/2}(t)-\left(\ell+d-3\right)C_{\ell-2}^{d/2-1/2}(t)=
\]
\[
 -\left(2\ell+d-3\right)\,(1-t)\, 
C_{\ell-1}^{d/2-1/2}(t)+\left(2\ell+d-3\right) 
C_{\ell-1}^{d/2-1/2}(t)-\left(\ell+d-3\right)C_{\ell-2}^{d/2-1/2}(t).
\]
We thus have
\begin{align*}
 kQ_{k,j}^d =\;&\left(2k+d-3\right) Q_{k-1,j}^d -\left(k+d-3\right)Q_{k-2,j}^d 
\\
 \qquad 
-\;&\left(2k+d-3\right)\int_{-1}^1(1-t^2)^{d/2-1}C_{k-1}^{d/2-1/2}(t)C_j^{
d/2-1/2}(t)\,dt,
\end{align*}
which implies for $k\geq2$:
\begin{equation}\label{eq:recursion}
 k\,Q_{k,j}^d =\begin{cases}\left(2k+d-3\right) Q_{k-1,j}^d 
-\left(k+d-3\right)Q_{k-2,j}^d  & k\neq j+1 \\ \left(2k+d-3\right) Q_{k-1,j}^d 
-\left(k+d-3\right)Q_{k-2,j}^d -\frac{\pi 
2^{3-d}\Gamma(d+k-2)}{\Gamma(k)\Gamma\left(\frac{d}{2}-\frac12\right)^2}& 
k=j+1\end{cases}
\end{equation}
These equalities together with $Q_{k,j}^d=Q_{j,k}^d$ and the values of 
$Q_{0,0}^d,\;Q_{1,0}^d$ and $Q_{1,1}^d$ define the value of $Q_{k,j}^d$ for all 
$k,j,d$. We finish the proof with five claims.

\begin{claim}\label{claim1}
For all $k\geq0$ we have $Q_{k,0}^d=Q_{0,0}^d$. 
\end{claim}
Indeed, we 
have 
already proved it for $k=1$. We now use induction, so we let $k\geq2$ and 
assume that the claim is true up to $k-1$. From \eqref{eq:recursion}, we have
\[
 Q_{k,0}^d=\frac{1}{k}\left(\left(2k+d-3\right) Q_{k-1,0}^d 
-\left(k+d-3\right)Q_{k-2,0}^d\right)=Q_{0,0}^d,
\]
as wanted.

\begin{claim}\label{claim2}
For all $k\geq1$ we have $Q_{k,1}^d=Q_{1,1}^d$. 
\end{claim}
Indeed, we have
\[
 Q_{2,1}^d=\frac{1}{2}\left(\left(d+1\right) Q_{1,1}^d 
-\left(d-1\right)Q_{0,1}^d -\frac{\pi 
2^{3-d}\Gamma(d)}{\Gamma\left(\frac{d}{2}-\frac12\right)^2}\right)=Q_{1,1}^d,
\]
where for the last equality we use \eqref{eq:legendre}. Again by induction on 
$k$ we assume that $k\geq3$ and the claim is true up to $k-1$. Then, from 
\eqref{eq:recursion}
\[
 Q_{k,1}^d=\frac{1}{k}\left(\left(2k+d-3\right) Q_{k-1,j}^d 
-\left(k+d-3\right)Q_{k-2,j}^d\right)=Q_{1,1}^d,
\]
as wanted.

\begin{claim}\label{claim3} 
The lemma holds for $0\leq k\leq j$.
\end{claim}
Indeed, from Claims~\ref{claim1} and \ref{claim2} we know this for $k=0,1$. 
Again using induction and \eqref{eq:recursion}, as long as $k\leq j$ we have
\begin{align*}
 Q_{k,1}^d=&\frac{1}{k}\left(\left(2k+d-3\right) Q_{k-1,j}^d 
-\left(k+d-3\right)Q_{k-2,j}^d\right)\\
=&\frac{Q_{0,0}^d}{k}\left(\left(2k+d-3\right)\binom{d+k-3}{k-1}
-\left(k+d-3\right)\binom{d+k-4}{k-2}\right)\\
 =&\binom{d+k-2}{k}Q_{0,0}^d,
\end{align*}
as wanted.

\begin{claim}\label{claim4}
For all $j\geq1$ we have $Q_{j+1,j}^d=Q_{j,j}^d$.
\end{claim}
Indeed, using 
\eqref{eq:recursion} and Claim~\ref{claim3} and denoting
\[
 R=\frac{\pi 
2^{3-d}\Gamma(d+j-1)}{\Gamma(j+1)\Gamma\left(\frac{d}{2}-\frac12\right)^2},
\]
we have
\begin{align*}
 Q_{j+1,j}^d=&\frac{Q_{0,0}^d}{j+1}\left(\left(2j+d-1\right) \binom{d+j-2}{j} 
-\left(j+d-2\right)\binom{d+j-3}{j-1} 
-R\right)\\=&\frac{Q_{0,0}^d}{j+1}\left(\left(j+d-1\right) \binom{d+j-2}{j} 
-R\right).
 \end{align*}
From \eqref{eq:legendre} we have
\begin{align*}
 R=\frac{\pi 
2^{3-d}\Gamma(d+j-1)}{\Gamma(j+1)\Gamma\left(\frac{d}{2}-\frac12\right)^2}
=
% &\frac{ 
% 2^{d-3}\Gamma(d+j-1)\Gamma\left(\frac{d}{2}-1\right)^2} 
% {\Gamma(j+1)\Gamma(d-2)^2 
% }\\=&\frac{2^{d-2}\Gamma(d+j-1)\Gamma\left(\frac{d}{2}-1\right)\Gammaf{d}{2}}{ 
% \Gamma(j+1)\Gamma(d-2)\Gamma(d-1)}\\=&\frac{\Gamma(d+j-1)}{ 
% \Gamma(j+1)\Gamma(d-2)}Q_{0,0}^d\\=&
(d-2)\binom{d+j-2}{j}Q_{0,0}^d.
\end{align*}
We have then proved that
\[ 
Q_{j+1,j}^d=\frac{Q_{0,0}^d}{(j+1)}\binom{d+j-2}{j}
\left(j+d-1-(d-2)\right)=\binom{d+j-2}{j}Q_{0,0}^d=Q_{j,j}^d,
\]
as claimed.

\begin{claim} \label{claim5}
For all $\ell\geq0$ we have $Q_{j+\ell,j}=Q_{j,j}$.
\end{claim}
Indeed, from Claim~\ref{claim4} the equality holds for $\ell=1$. Reasoning by 
induction on $\ell$, assume that the equality holds up to $\ell-1$. From 
\eqref{eq:recursion} and Claims~\ref{claim3} and \ref{claim4} we have
\[
 Q_{j+\ell,j}=\frac{1}{j+\ell}\left(\left(2j+2\ell+d-3\right) Q_{j,j}^d 
-\left(j+\ell+d-3\right)Q_{j,j}^d\right)=Q_{j,j}^d.
\]
This finishes the proof of Claim~\ref{claim5} and of the lemma.

% Note that if $Q_{k,j}^d=Q_{k+1,j}^d$ then the recurrence implies 
% $Q_{k,j}^d=Q_{k',j}^d$ for all $k'\geq1$, so it suffices for the first part 
% to prove that
% \[
%  Q_{k,k}^d=Q_{k+1,k}^d.
% \]
% 
% We now prove the second equality. From the recurrences above,
% \begin{align*}
%  k\,Q_{k,k}^d=\;&\left(2k+d-3\right) Q_{k-1,k}^d 
%-\left(k+d-3\right)Q_{k-2,k}^d\\
%  =\;&\left(2k+d-3\right) Q_{k-1,k-1}^d -\left(k+d-3\right)Q_{k-2,k-2}^d.
% \end{align*}
% That is, the sequence 
% \[
% x_k^d=\frac{Q_{k,k}^d }{2^{d-2}\mathrm 
% B\left(\frac{d}{2},\frac{d}{2}-1\right)}
% \]
% satisfies
% \[
%  x_0^d=1,\quad x_1^d=d-1,\quad 
% k\,x_{k}^d=(2k+d-3)x_{k-1}-(k+d-3)x_{k-2},\quad k\geq2.
% \]
% It is straight forward to check that this implies:
% \[
%  x_k^d=\binom{d+k-2}{k}.
% \]
% The last equality of the lemma then follows.
\end{proof}

%===============================================================================
%===============================================================================
%===============================================================================

\begin{proof}[Proof of Theorem~\ref{th:case2}]
Note that by reordering the terms:
\begin{align*}
 \sum_{j,k=0}^\infty a_ka_j 
Q_{k,j}^{d}=\;&\sum_{\ell=0}^\infty\left(a_\ell^2Q_{\ell,\ell}^d+\sum_{j>\ell}
a_\ell a_j Q_{\ell,j}^{d}+\sum_{k>\ell}a_ka_\ell Q_{k,\ell}^{d}\right)\\
 \underset{\text{Lemma \ref{lem:Q}}}{=}\;&
 Q_{0,0}^d\sum_{\ell=0}^\infty 
a_{\ell}\binom{d+\ell-2}{\ell}\left(a_\ell+2\sum_{j>\ell} a_j \right).
 \\
 % \underset{\eqref{eq:nK1}}{=}\;&Q_{0,0}^d\sum_{k=0}^\infty 
%a_k\left(\sum_{j<k}a_j\binom{d+j-2}{j}+\binom{d+k-2}{k}\sum_{j\geq k}a_j\right)
\end{align*}
 We then have
 \[
  \mathbb{E}_{x\in(\Sd)^n}(E_2(x))= \frac{\omega_{d-1}}{2\omega_d}\left(n^2 
2^{d-2} {\mathrm B}\left(\frac{d}{2},\frac{d}{2}-1\right)-\sum_{j,k=0}^\infty 
a_ka_j Q_{k,j}^{d}\right)=
\]
\[
 \frac{\omega_{d-1}Q_{0,0}^d}{2\omega_d}\left(n^2 -\sum_{\ell=0}^\infty 
a_{\ell}\binom{d+\ell-2}{\ell}\left(a_\ell+2\sum_{j>\ell} a_j \right)\right).
\]
Use \eqref{eq:Q00omegas} to finish the proof.
\end{proof}

%===============================================================================
%===============================================================================
%===============================================================================

The following gives an alternative formula for the expected value computed in 
Theorem~\ref{th:rieszsharmonic} for the case $s=2$ as well as an asymptotic 
estimate. Note that the harmonic kernel $K_L$ is obtained when in the general 
setting of this section we let
\begin{equation}\label{eq:armonicas}
  a_k=\begin{cases}\frac{2k+d-1}{d-1} &k\leq L\\0&k>L\end{cases}.
 \end{equation}
 
From Theorem~\ref{th:case2} we readily have:

\begin{corollary}\label{cor:s2harmonicassymptotics}
 Let $x=(x_1,\ldots,x_n)\in (\S^d)^n$ be $n=\pi_L$ points generated by the determinantal 
random point process associated to harmonic kernel $K_L.$ Then, 
$\mathbb{E}_{x\in(\S^d)^n}(E_2(x))$ equals
 \[
 V_2(\S^d)\left(n^2 -\sum_{\ell=0}^L 
\frac{2\ell+d-1}{d-1}\binom{d+\ell-2}{\ell}\left(\frac{2\ell+d-1}{d-1}+2\sum_{
j>\ell} \frac{2j+d-1}{d-1} \right)\right).
 \]
\end{corollary}
 \begin{remark}
 It is possible to recover  the asymptotic estimate of 
Theorem~\ref{th:asymptoticrieszsharmonic} from 
Corollary~\ref{cor:s2harmonicassymptotics} in the case $s=2$, that is
 \[  
\mathbb{E}_{x\in(\S^d)^n}(E_2(x))=V_2(\S^d)\left(n^2-\frac{4n^{1+2/d}}{
(d+2)(d-1)}\left(\frac{d!}{2}\right)^{2/d}\right)+o(n^{1+2/d}).
 \] 
 \end{remark}

%===============================================================================
%===============================================================================
%===============================================================================

\subsubsection{Optimality of the harmonic kernel. Theorem~\ref{th:optimal}} We now 
prove that the harmonic kernel gives optimal values of the expected $2$-energy 
among rotationally invariant kernels.

\begin{proof}[Proof of Theorem~\ref{th:optimal}]
 Let $r\in\N$ be such that $a_j=b_j=0$ for $j\geq r$, and assume that $a\neq 
b$. Note from \eqref{eq:esperanzaQs} and Lemma \ref{lem:Q} that 
$\mathbb{E}_a<\mathbb{E}_b$ is equivalent to $F(a)\geq F(b)$ where
 \[
  F(x)=x^TMx,
 \]
 where $M=M_r$ is the symmetric matrix given by
 \[   
M=\begin{pmatrix}\binom{d+\min(k,j)-2}{\min(k,j)}_{k,j=0,\ldots,r} 
\end{pmatrix}=\begin{pmatrix}\binom{d-2}{0} &\binom{d-2}{0}& 
\binom{d-2}{0}&\cdots&\binom{d-2}{0}\\ 
\binom{d-2}{0}&\binom{d-1}{1}&\binom{d-1}{ 
1}&\cdots&\binom{d-1}{1}\\\binom{d-2}{0}&\binom{d-1}{1}&\binom{d}{2} 
&\cdots&\binom{d}{2}\\\vdots&\vdots&\vdots&\ddots&\vdots\\ 
\binom{d-2}{0}&\binom{d-1}{1}&\binom{d}{2}&\cdots&\binom{d+r-2}{r}\end{pmatrix}
\in\Z^{(r+1)\times (r+1)}.
 \]
 We also consider the vector 
 \[
V=\left(\binom{d-2}{0},\binom{d-1}{1},\binom{d}{2},\ldots,\binom{d+r-2}{r}
\right)\in\Z^{r+1},
 \]
and, for $0\leq i<j\leq r$ we let $w_{ij}\in\R^{r+1}$ be the vector all of whose components 
are zero except for the $i$th component and the $j$th component that satisfy:
 \[
  (w_{ij})_i=\binom{d+i-2}{i}^{-1},\quad (w_{ij})_j=-\binom{d+j-2}{j}^{-1}.
 \]
Note that, for coherence in the exposition, we are numbering the entries of $w_{ij}$ from $0$ to $r$ instead of doing it from $1$ to $r+1$. Then,
\begin{equation}\label{eq:wV}
 w_{ij}^TV=0,\quad \forall\,i,j,\;0\leq i< j\leq r.
\end{equation}
 An elementary computation shows that if $i<j$ then all the components of the 
vector $Mw_{ij}\neq0$ are positive or zero. We thus have
 \[
  x^TMw_{ij}\geq0\quad\text{ for all } x\in[0,\infty)^{r+1}.
 \]
Note now that if $x\in[0,\infty)^{r+1}$ then for all $i,j$ we have
\begin{equation}\label{eq:DF}
 DF(x)(w_{ij})=w_{ij}^TMx+x^TMw_{ij}=2x^TMw_{ij}\geq0.
\end{equation}
Moreover, if $x_i>0$ then $DF(x)(w_{ij})>0$ and the function is strictly 
increasing in that direction. We will now construct a sequence
\[
 b=x^0,x^1,\ldots,x^t=a
\]
with the property that $x^k-x^{k-1}$ is a non-negative multiple of $w_{ij}$ 
for some $i,j$ with  $i<j$. Although the coordinates of $x^0,x^t$ have a 
particular form given by \eqref{eq:form}, the coordinates of $x^k$ are just 
non-negative real numbers. For the construction, let $i$ be the first index 
such that $a_i\neq0,b_i=0$ and let $j$ be the greatest index such that 
$b_j\neq0$ and $a_j=0$ (these $i,j$ exist because $a\neq b$ and the hypotheses 
$\sum a_i\binom{d+i-2}{i}=\sum b_i\binom{d+i-2}{i}=n$). Note also from 
\eqref{eq:if} that necessarily $j>i$. Then, we define
\[
 x^1=x^0+\left(1+\frac{2i}{d-1}\right)\binom{d+i-2}{i}w_{ij}\in[0,\infty)^{r+1},
\]
and in general to construct $x^{k+1}$ from $x^k$ we let $i$ be the smallest 
index, among those such that $a_i\neq0$, such that $x^k_i\neq a_i=1+2i/(d-1)$ 
and $j$ the greatest index such that, $x^k_j>0$ and $a^k_j=0$, if those indices 
exist. It can be easily seen by induction that \eqref{eq:if} is also satisfied 
changing $b$ to $x^k$ and thus, if such $i,j$ exist, we have $i<j$. Then, we let
\[
 x^{k+1}=x^k+\lambda w_{ij}, \quad 
\lambda=\min\left(\binom{d+j-2}{j}x^k_j,\left(1+\frac{2i}{d-1}
-x^k_i\right)\binom{d+i-2}{i}\right).
\]
% See figures \ref{fig:closingD} and \ref{fig:refillingD} for a dynamical 
% picture of the two different alternatives (according to the value of the 
% minimum).
% 
% \begin{figure}
%     \begin{center}
%            \includegraphics[width=0.9\textwidth]{closing1.pdf}
%            \centering
%            \includegraphics[width=0.9\textwidth]{closing2.pdf}
%           \caption{Initial and final states $x^k$ and $x^{k+1}$ in the case 
% that $\lambda=\binom{d+j-2}{j}x^k_j$. The ``mass transfer'' stops when the 
% }\label{fig:closingD}
%     \end{center}
% \end{figure}
% 
% 
% \begin{figure}
%     \begin{center}
%            \includegraphics[width=0.9\textwidth]{refilling1.pdf}
%            \centering
%            \includegraphics[width=0.9\textwidth]{refilling2.pdf}
%            \caption{Initial and final states $x^k$ and $x^{k+1}$ in the case 
% that 
% $\lambda=\left(1+\frac{2i}{d-1}-x^k_i\right)\binom{d+i-2}{i}$}\label{
% fig:refillingD}
%     \end{center}
% \end{figure}

From \eqref{eq:DF} and the comment after \eqref{eq:DF} it is clear that
\[
F(b)=F(x^0)<F(x^1)\leq F(x^2)\leq \cdots\leq F(x^k),\quad \forall\;k\in\N.
\]
We just have to prove that the sequence satisfies $x^k=a$ for some $k\in\N$. 
First note that from \eqref{eq:wV} we have for all $k\in\N$:
\begin{equation}\label{eq:xk}
 (x^k)^T V=(x^{k-1})^T V=\cdots=x^ 0 V=n,\quad\text{that is,}\quad \sum_{i=0}^r 
x^k_i\binom{d+i-2}{i}=n.
\end{equation}
% It is also clear from the construction that the hypotheses on $b$ is also 
% satisfied by $x^k$, that is 
%  \begin{equation}\label{eq:hyp2}
%   \sharp\{i\leq k:a_i\neq0\}\geq\sharp\{i\leq k:x^k_i\neq0\},\quad 
% \forall\,k\geq0.
%  \end{equation}
Moreover, by construction we have $x^k_i\in [0,1+2i/(d-1)]$ for all $k$ and 
$i$. On the other hand, the sequence can only finish if some of the indices 
$i,j$ in the construction does not exist. Namely, the sequence stops when:
\begin{enumerate}
 \item[(i)] For every $i$ such that $a_i\neq0$ we have $x^k_i=a_i$, or
 \item[(ii)] For every $j$ such that $x^k_j\neq0$ we have $a_j>0$.
\end{enumerate}
From \eqref{eq:xk} and the similar equality $\sum_{i=0}^r 
a_i\binom{d+i-2}{i}=n$ it is clear that these two conditions are equivalent and 
actually imply $x^k=a$. This proves that, if the sequence finishes, then its 
last element is equal to $a$. Now note that at each iteration $x^k\mapsto 
x^{k+1}$, either one coordinate of $x^{k}$ is set to $0$ (and in this case, 
that 
coordinate remains untouched in further iterations), or one of them is set to 
its maximum value $1+2i/(d-1)$ (which, again, implies that this coordinate 
remains untouched in further iterations). So the number of iterations of the 
sequence is at most the total number of coordinates in $a$ or $b$, that is at 
most $r+1$. This finishes the proof of the theorem.
\end{proof}

\begin{figure}
    \begin{center}
           \includegraphics[width=0.9\textwidth]{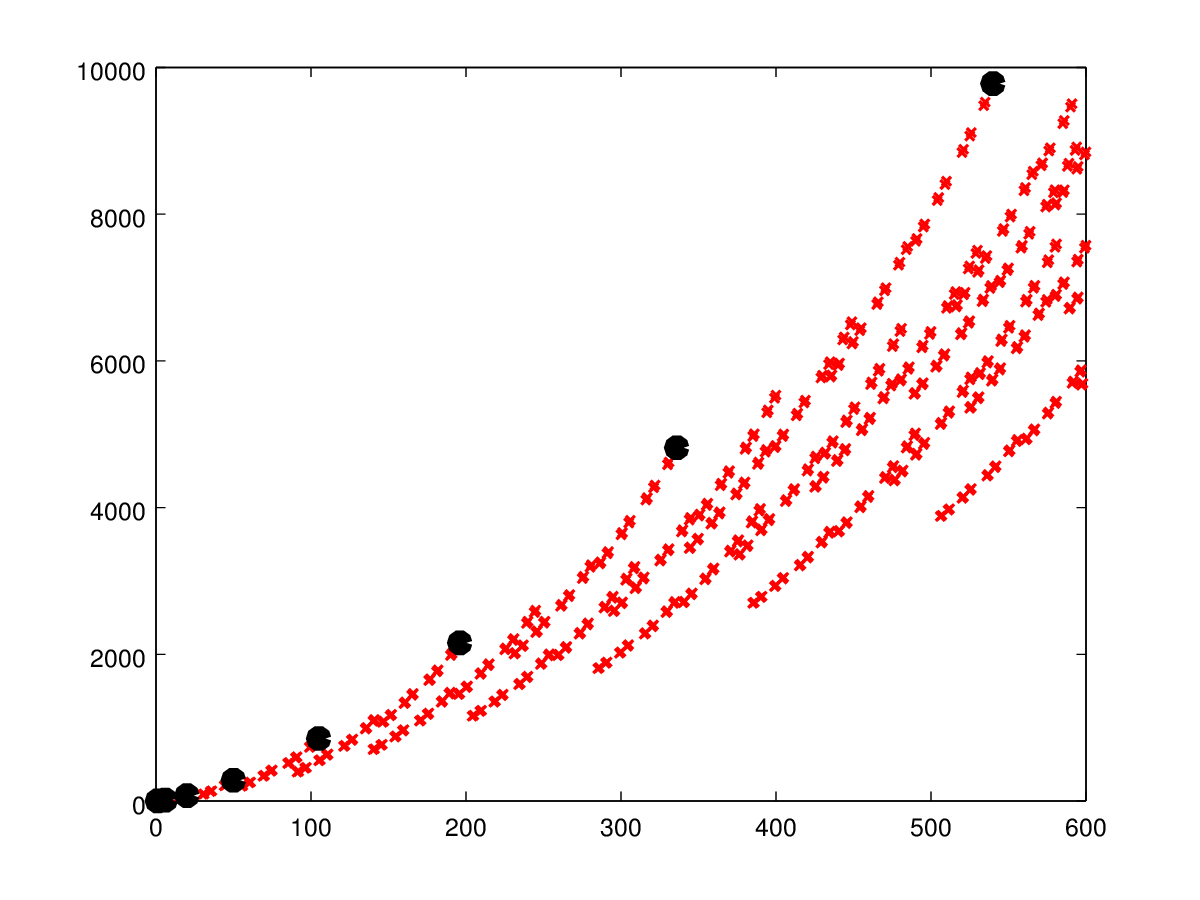}
          \caption{Value of $\sum_{\ell=0}^\infty 
a_{\ell}\binom{d+\ell-2}{\ell}\left(a_\ell+2\sum_{j>\ell} a_j \right)$ for 
different kernels in dimension $d=4$. Black marks correspond to the harmonic 
kernels, i.e. when the $a_k$'s are given by \eqref{eq:armonicas},  while red 
marks correspond to other kernels. For some values of $n$ (for example, $n=4$) 
there is no kernel that attains this number of points. For other values of $n$ 
(for example, for $n=6$) there is only one such kernel. For yet another 
collection of values of $n$ (including for example $n=196$ and $n=540$) there 
are several choices of kernels which produce this number of points. In these 
two particular values, one of the choices corresponds to the harmonic kernel. 
The optimality of the harmonic kernel proved in Theorem~\ref{th:optimal} is 
clearly visible: the sum is maximal when the harmonic kernel is used, hence 
from 
Theorem~\ref{th:case2} the expected value of $E_2$ is 
minimal.}\label{fig:Casodiguala4}
    \end{center}
\end{figure}
\begin{figure}
    \begin{center}
           \includegraphics[width=0.9\textwidth]{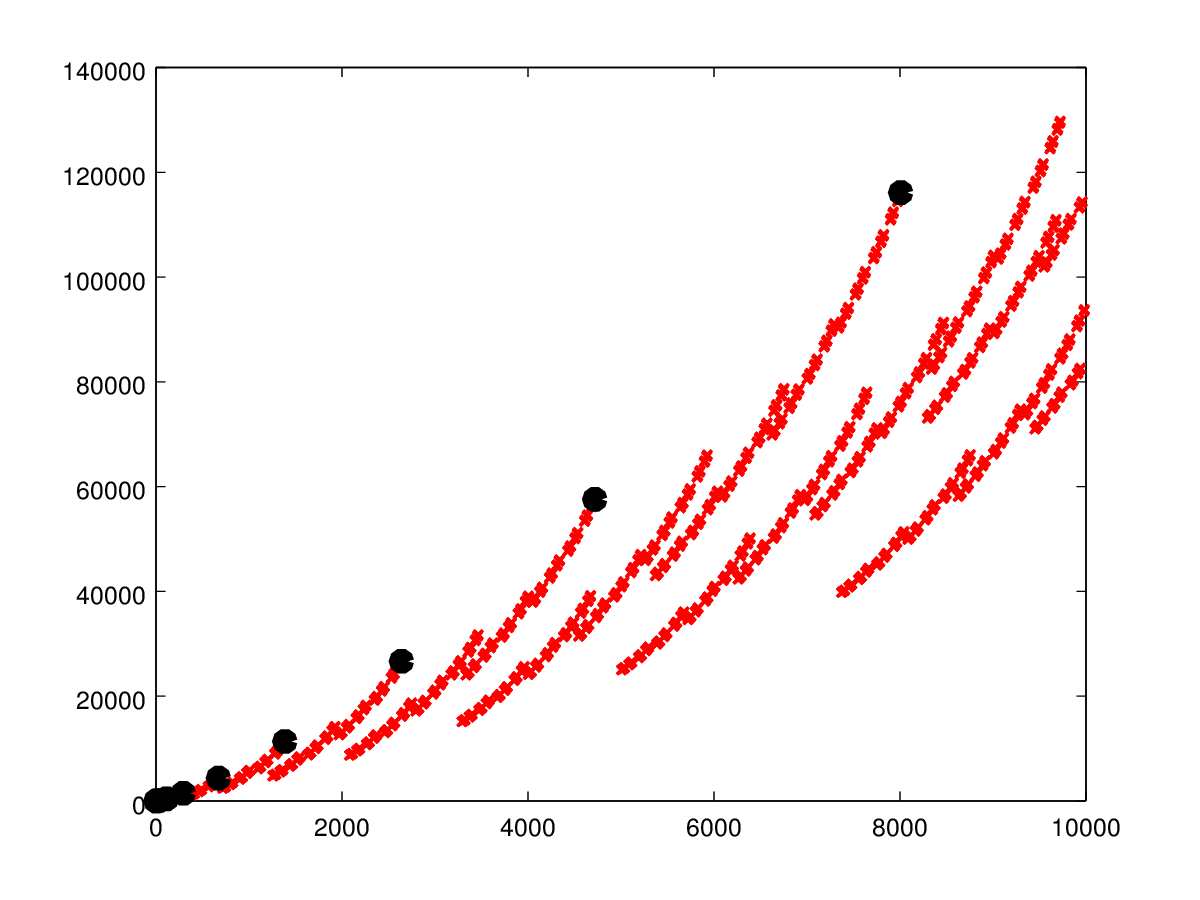}
          \caption{Same as Figure \ref{fig:Casodiguala4} but for dimension 
$d=6$. Again the optimality of the harmonic kernel is clearly 
visible.}\label{fig:Casodiguala6}
    \end{center}
\end{figure}

\subsection{Linear Statistics. Propositions~\ref{prop:roughvariance} and \ref{prop:smooth}}

The objective is to estimate the asymptotic behavior of the variance of the number of points of the harmonic ensemble
to be found in a spherical cap. To get this estimate we study the trace
  $\tr (\mathcal{K}_A-\mathcal{K}_A^2)$, where $\mathcal{K}_A$ is the 
integral operator of concentration on the spherical cap $A\subset \S^d$,
\[
\mathcal{K}_A Q(u)=\int_A Q(v) K_L(u,v)d\mu(v),\;\;Q\in \Pi_L.
\]

  The proof of Proposition~\ref{prop:roughvariance} is
  similar to \cite[Proposition 3.1.]{Mar07}. The idea in that paper was, following 
Landau's work, to study the density of discrete sets (Marcinkiewicz-Zygmund and interpolating arrays) 
relating the density with spectral properties of the concentration operator
in ``small'' spherical caps. The hypothesis are now different (we consider ``big'' spherical caps also) so we will sketch the proof.

\begin{proof}[Proof of Proposition~\ref{prop:roughvariance}]
   It is clear that the variance of the random variable, $n_A,$ counting the 
number of points in $A,$ is invariant by rotations of $A$, because the process 
is also invariant, so to compute $\operatorname{Var}(n_A)$ we assume that 
$A=B(\mathbf{n},\theta_L)$ with $\mathbf{n}$ being the north pole and 
$\theta_L\in[0,\pi]$. Denote $\theta_L=\alpha_L/ L$ with $\alpha_L=O(L)$, and
$\alpha_L\to \infty$ when $L\to \infty$. The following formula to compute the 
covariance can be found in 
\cite[Formula~28]{RyderVirag07}
\begin{equation}				\label{covariance_formula}
  \operatorname{Cov}(f,g)=\frac{1}{2}\int_{\S^d} \int_{\S^d} 
(f(x)-f(y))(g(x)-g(y)) 
|K_L(x,y)|^2 d\mu(x)d\mu(y),
\end{equation}
  for bounded $f,g$. In particular
\[
\operatorname{Var}(n_A)=\operatorname{Cov}(\chi_A,\chi_A)=\int_{A}\int_{A^c} |K_L(x,y)|^2 
d\mu(x)d\mu(y)=\tr (\mathcal{K}_A-\mathcal{K}_A^2).
\]

By rotation invariance 
\begin{align*}
  \int_{A} & \int_{A^c} |K_L(x,y)|^2 d\mu(x)d\mu(y) \\
= & A_{d,L}^2\int_0^{\theta_L} \sin^{d-1} 
\eta \left(
\int_{\theta_L-\eta}^\pi |P_L^{(1+\lambda,\lambda)}(\cos \theta)|^2 \sin^{d-1} 
\theta d\theta \right) d\eta,
\end{align*}
where
\[
 A_{d,L}=\frac{\pi_L \omega_{d-1}}{\binom{L+\frac{d}{2}}{L} 
\omega_d}=\frac{2^{1-d}}{\Gamma\left( \frac{d}{2} \right)}L^{d/2}+o(L^{d/2}).
\]
For some fixed $c>0,$ we split the inner integral above in three 
summands
corresponding to
\[
\left\{ \frac{c}{L}\le \theta \le \pi-\frac{c}{L} \right\}=\mbox{I},\;\; \left\{ \theta > \pi-\frac{c}{L} \right\}=\mbox{II},
 \;\;\mbox{and}\;\;\left\{ \frac{c}{L} < \theta \right\}=\mbox{III}.
\]

  The integral in $\mbox{I}$ can be bounded above, by using 
classical estimates for 
Jacobi polynomials \eqref{Jacobiestimate}
\begin{align*}
    \frac{A_{d,L}^2}{L^d} &  \int_{0}^{\alpha_L} \eta^{d-1} 
\int_{\max(\frac{\alpha_L-\eta}{L},\frac{c}{L})}^{\pi-\frac{c}{L}}
|P_L^{(1+\lambda,\lambda)}(\cos \theta)|^2 \sin^{d-1} \theta d\theta d\eta
\\
&
\le 
\frac{A_{d,L}^2 2^{d-1}}{L^d}  \int_{0}^{\alpha_L} \eta^{d-1} 
\int_{\max(\frac{\alpha_L-\eta}{L},\frac{c}{L})}^{\pi-\frac{c}{L}}
\frac{1}{\pi L \sin^2 \frac{\theta}{2}} d\theta d\eta
\\
&
=
\frac{A_{d,L}^2 2^{d-1}}{\pi L^{d+1}}\int_0^{\alpha_L} \eta^{d-1}  \left( 
\cot\left( 
\frac{\max\{\alpha_L-\eta,c\}}{2L} \right) -\tan\left( \frac{c}{2 L} \right) 
\right) d\eta.
\end{align*} 
The main term comes from the first summand of this last integral 
because the second is of order $L^{-2}\alpha_L^d$. For the first summand we 
split the integral again and do a change of variables
\begin{align}				\label{provis}
\int_0^{\alpha_L} & \eta^{d-1}   \cot\left( 
\frac{\max\{\alpha_L-\eta,c\}}{2L} \right) 
 d\eta\nonumber
\\
&
= 
 \int_{c}^{\alpha_L}  (\alpha_L-\eta)^{d-1} 
\cot\left(\frac{\eta}{2L}\right)  d\eta+
  \int_{\alpha_L-c}^{\alpha_L} \eta^{d-1} \cot\left( 
\frac{c}{2L} \right) d\eta,
\end{align}  
  and that the second summand in above is of order $L \alpha_L^{d-1}$. 
  For the first summand in (\ref{provis}) we expand the polynomial in $\eta$ and use the 
estimate
  $x \cot x\le 1$ for $x\in [0,\pi/4]$ 
  to get the bound
\begin{align*}
  \alpha_L^{d-1} & \int_{c}^{\alpha_L} 
\cot\left(\frac{\eta}{2L}\right)  d\eta+L O(\alpha_L^{d-1})=
2 L \alpha_L^{d-1}  \log 
\frac{\sin(\frac{\alpha_L}{2L})}{\sin(\frac{c}{2L})}+L O(\alpha_L^{d-1})
\\
&
=
2 L \alpha_L^{d-1} \log \alpha_L+L O(\alpha_L^{d-1}).
\end{align*}
  For the integral in $\mbox{II}$ we bound the Jacobi polynomial by $C 
L^{2\lambda}$ getting a term of order $L^{-2} \alpha^d_L=O(\alpha^{d-2}_L)$. 
 Finally, in $\mbox{III}$ we bound the Jacobi polynomial by its maximum $C 
L^{d/2}$ getting another term $O(\alpha^{d-1}_L)$.
 
 Putting all together we get
\[
\operatorname{Var}(n_A)\le \frac{2^{2-d}}{\pi 
\Gamma\left(\frac{d}{2}\right)^2}\alpha_L^{d-1}\log 
\alpha_L+O(\alpha_L^{d-1}).
\]
\end{proof}

%===============================================================================
%===============================================================================
%===============================================================================

We now turn to the smooth case.

\begin{proof}[Proof of Proposition~\ref{prop:smooth}]
Given a bounded function $\phi:\S^d\rightarrow \mathbb R$ we denote by 
$T_{\phi}$ the
Toeplitz operator on $\Pi_L$ with symbol
$\phi$, i.e. $T_{\phi}(h):=\K_{L}(\phi h)$ where $\K_{L}$ denotes
the orthogonal projection from $L^{2}(\S^d)$ to $\Pi_L$
\[
\K_{L} f(x)=\int_{\S^d} K_L(x,y)f(y)d\mu(y),\;\;f\in L^2(\S^d),
\]
i.e. $T_{\phi}$ is the self-adjoint operator on $\Pi_L$ determined
by 
\[
\left\langle T_{\phi} P, P \right\rangle =\left\langle 
\phi P,P \right\rangle 
\]
for any $P\in \Pi_L$. Then it follows from \eqref{covariance_formula}
% \begin{equation}\label{trace}
% \operatorname{Var}(\chi_\phi)
% =\frac{1}{2}\int_{\S^d\times 
% \S^d}\left|K_{L}(x,y)\right|^{2}\left(\phi(x)-\phi(y)\right)^{2}.
% \end{equation}
that for
$\phi$ Lipschitz 
\[
\operatorname{Var}(\mathcal{X}(\phi))
\lesssim
\frac{1}{2}\int_{\S^d\times 
\S^d}\left|K_{L}(x,y)\right|^{2} d^2(x,y) d\mu(x)d\mu(y).
\]
Now, setting $\psi(x)=x_{i}$ for a fixed index $i\in\{1,\ldots,d\}$, we get
\[
\operatorname{Var}(\mathcal{X}(\phi))
\lesssim
\int_{\S^d\times 
\S^d}\left|K_{L}(x,y)\right|^{2}\left(\psi(x)-\psi(y)\right)^{2}d\mu(x)d\mu(y).
\]
On the other hand, an elementary computation shows that 
\[
 \frac{1}{2}\int_{\S^d\times 
\S^d}\left|K_{L}(x,y)\right|^{2}\left(\psi(x)-\psi(y)\right)^{2}
d\mu(x)d\mu(y)=\tr 
T_{\psi}^{2}-\tr T_{\psi^{2}}.
\]

We note
that there exists a vector subspace $V_{L}$ in $\Pi_L$ with dimension
$\dim V_L=\pi_L-O(L^{d-1})$ such that when restricted to $V_L$, 
$T_{\psi}(P)=P\psi $ and $T_{\psi}^{2}(P)=P\psi^{2}$.
We can take $V_{L}$ to be the space spanned by the restrictions to
$\S^d$ of all polynomials of total degree at most $L-2$, i.e. $\Pi_{L-2}$. 
% The dimension of $\pi_L=\dim(\Pi_L)=\frac {d+2L}{d}\binom{d+L-1}{L}$.

If we denote by $W_{L}$ the orthogonal complement
of $V_{L}$ in $\Pi_L$ then $\dim(W_L)=\pi_L-\pi_{L-2}=O(L^{d-1})$.  
Setting 
$A_{L}:=T_{\psi}^{2}-T_{\psi^{2}}$
gives $A_{L}=0$ on $V_{L}$ and hence
\[
\tr T_{\psi}^{2}-\tr T_{\psi^{2}}=0+\tr A_{L|_{W_{L}}}\lesssim 
L^{d-1},
\]
using that $\left\langle T_{\psi} P,P \right\rangle 
\leq\left\langle P,P \right\rangle\sup_{\S^d}|\psi| $
and $\dim W_{L}=O(L^{d-1})$.

In the other direction if we take $\phi(x)=x_i$, then
\[
 \operatorname{Var}(\mathcal{X}(\phi))=
\frac{1}{2}\int_{\S^d\times 
\S^d}\left|K_{L}(x,y)\right|^{2} (\phi(x)-\phi(y))^2d\mu(x)d\mu(y).
\]
Therefore (recall that $\mbox{\bf n}$ stands for the north pole)
\[
 \operatorname{Var}(\mathcal{X}(\phi))
\gtrsim
\int_{\S^d }\left|K_{L}(x,\mbox{\bf n})\right|^{2} d^2(x,\mbox{\bf n})d\mu(x)  \simeq
\int_0^\pi \sin^{d-1}(\theta) |\theta|^2 |P_L^{(1+\lambda,\lambda)}(\cos 
\theta)|^2d\theta. 
\]
Using the classical estimates for Jacobi polynomials as before we get
$\operatorname{Var}(\mathcal{X}(\phi))\gtrsim L^{d-1}$.
\end{proof}

%=============================================================================
%=============================================================================
%=============================================================================

\subsection{Separation distance. Proof of Proposition \ref{prop:G}}
Apply Proposition \ref{th:determinantal} to $f(u,v)=\mathbf{1}_{\|u-v\|\leq 
t}$, which yields
\[
 2\mathbb{E}_{x\in (\S^d)^n}(G(t,x))=\int_{u,v\in\S^d,\|u-v\|\leq 
t}(K_L(u,u)^2-|K_L(u,v)|^2)\,d\mu(u)\,d\mu(v).
\]
The integrand depends only on the scalar product $\langle u,v\rangle$ so by rotation invariance we have
\begin{align*} 
2\mathbb{E}_{x\in (\S^d)^n}(G(t,x))=&\frac{\pi_L^2\omega_{d-1}}{\omega_d}\int_{1-t^2/2}
^1\left(1-\frac{P_L^{(1+\lambda,\lambda)}(s)^2}{\binom{L+d/2}{L}^2}
\right)\left(1-s^2\right)^{d/2-1}\,ds\\=&\frac{\pi_L^2\omega_{d-1}}{\omega_d}
\int_0^{t^2/2}\left(1-\frac{P_L^{(1+\lambda,\lambda)}(1-s)^2}{\binom{L+d/2}{L}^2}
\right)\left(2s-s^2\right)^{d/2-1}\,ds.
\end{align*}
%(The last by a change of variables $s\to 1-s$). 
From 
Lemma~\ref{lemma:jacobibound} below we conclude
\begin{align*}
\mathbb{E}_{x\in (\S^d)^n}(G(t,x))\leq&\frac{\pi_L^2\omega_{d-1}}{2\omega_d}\int_0^{t^2/2} 
\frac{2(L^2+Ld)}{d+2}s\left(2s\right)^{d/2-1}\,ds=\frac{L(L+d)\pi_L^2\omega_{d-1
}}{2(d+2)^2\omega_d}t^{d+2},
\end{align*}
as claimed.

%======================================================================================

We have used the following elementary lemma bounding Jacobi polynomials, a 
short proof is included for completeness:
\begin{lemma}\label{lemma:jacobibound}
 Let $L,d\geq1$. Then, for all $s\in\R$, $0\leq s\leq \frac{d+6}{(2L+d)L}$ we 
have
 \[
  1-\frac{L^2+Ld}{d+2}s\leq \frac{P_L^{(1+\lambda,\lambda)}(1-s)}{\binom{L+d/2}{L}}\leq 1-\frac{L^2+Ld}{d+2}s+k_0s^2,
 \]
 for some constant $k_0\in(0,\infty)$. In particular,
 \[
\frac{2(L^2+Ld)}{d+2}s-k_0s^2\leq  1-\frac{P_L^{(1+\lambda,\lambda)}(1-s)^2}{\binom{L+d/2}{L}^2}\leq  
\frac{2(L^2+Ld)}{d+2}s. 
 \]
\end{lemma}
\begin{proof}
 Let $q(s)=P_L^{(1+\lambda,\lambda)}(1-s)$. The expansion of $q$ in  the standard 
monomial basis is easy to compute from the derivatives for $0\leq k\leq L$ (see 
for example \cite[p.~1008]{losrusos}):
 \[  
\frac{d^k}{ds^k}q(0)=(-1)^k\frac{d^k}{ds^k}P_L^{(1+\lambda,\lambda)}(1)=\frac{
(-1)^k\Gamma(L+k+d)}{2^k\Gamma(L+d)}\binom{L+d/2}{L-k}.
 \]
 We thus have for $s\in\R$:
 \begin{align*}
\frac{q(s)}{\binom{L+d/2}{L}}=&\sum_{k=0}^L\frac{(-1)^k\Gamma(L+k+d)}{2^k k! 
\Gamma(L+d)}\frac{\binom{L+d/2}{L-k}}{\binom{L+d/2}{L}}\; s^k \\=& 
\sum_{k=0}^L\frac{(-1)^k\Gamma(L+k+d)\Gamma(L+1)\Gamma(1+d/2)}{2^k k! 
\Gamma(L+d)\Gamma(L-k+1)\Gamma(k+1+d/2)}\; s^k \\=&1-\frac{L^2+Ld}{d+2}s+R,
 \end{align*}
 where $R$ stands for the terms in the summation from $k=2$ to $k=L$. We will 
show that $R\geq0$ which finishes the proof. The terms of $R$, with the possible 
exception of the last one if $L$ is even, have alternating signs. It is then 
sufficient to show that for $k=2,4,6,\ldots$, $k<L$, the $k$th term in the 
summation is larger than the (absolute value of the) $(k+1)$th term, that is we 
have to show that for those values of $k$,
\[
 \begin{split}
  \frac{\Gamma(L+k+d)\Gamma(L+1)\Gamma(1+d/2)s^k}{2^k k! 
\Gamma(L+d)\Gamma(L-k+1)\Gamma(k+1+d/2)}\geq\\
  \frac{\Gamma(L+k+d+1)\Gamma(L+1)\Gamma(1+d/2)s^{k+1}}{2^{k+1} (k+1)! 
\Gamma(L+d)\Gamma(L-k)\Gamma(k+2+d/2)}.
 \end{split}
\]
 This is satisfied whenever
 \[
  s\leq \frac{(k+1) (2k+2+d)}{(L-k)(L+k+d)}.
 \]
 It is a trivial exercise to check that the hypotheses on $s$ guarantee this 
last inequality.
\end{proof}

\end{document}